\documentclass{article}
\usepackage{amsfonts,amsmath,amssymb,amsthm,cite,enumerate,graphicx,latexsym,mathrsfs}
\usepackage{authblk}
\usepackage{hyperref}
\usepackage[T1]{fontenc}
\usepackage[utf8]{inputenc}
\usepackage{float}
\usepackage{geometry}
\date{}
\makeatletter

\newcommand{\Rmnum}[1]{\expandafter\@slowromancap\romannumeral #1@}
\makeatother

\numberwithin{equation}{section}
\newtheorem{corollary}{Corollary}[section]

\newtheorem{lemma}{Lemma}[section]
\newtheorem{proposition}{Proposition}[section]

\newtheorem{theorem}{Theorem}[section]

\newcommand\theref[1]{Theorem~\ref{#1}}
\newcommand\lemref[1]{Lemma~\ref{#1}}

\newcommand\proref[1]{Proposition~\ref{#1}}

\newcommand\secref[1]{Section~\ref{#1}}

\title{Subsonic time-periodic solution to compressible Euler equations with damping in a bounded domain}
\author[a]{Peng Qu\thanks{{e}-mail: pqu@fudan.edu.cn}}
\author[b]{Huimin Yu\thanks{{e}-mail: hmyu@sdnu.edu.cn}}
\author[b]{Xiaomin Zhang\thanks{{e}-mail: zxm15924687@163.com}}
\affil[a]{School of Mathematical Sciences, Shanghai Key Laboratory for Contemporary Applied Mathematics, Fudan  University, Shanghai 200433 China}
\affil[b]{Department of mathematics, Shandong Normal University, Jinan 250014 China}
\begin{document}
\begin{sloppypar}
\date{}
\maketitle
\begin{center}
\begin{minipage}{130mm}{\small
\textbf{Abstract}:
In this paper, we consider the one-dimensional isentropic compressible Euler equations with linear damping $\beta(t,x)\rho u$ in a bounded domain, which can be used to describe the process of compressible flows through a porous medium.~And the model is imposed a dissipative subsonic time-periodic boundary condition.~Our main results reveal that the time-periodic boundary can trigger a unique subsonic time-periodic smooth solution which is stable under small perturbations on initial data. Moreover, the time-periodic solution possesses higher regularity and stability provided a higher regular boundary condition. 
}
\end{minipage}
\end{center}
\section{Introduction}
\indent\indent In this paper, we consider the one-dimensional compressible Euler equations for isentropic flow with linear damping in Eulerian coordinates
\begin{equation}\label{a1}
\left\{\begin{aligned}
&\partial_{t}\rho+\partial_{x}(\rho u)=0,\\
&\partial_{t}(\rho u)+\partial_{x}(\rho u^{2}+p)=\beta(t,x)\rho u,
\end{aligned}\right.
\end{equation}
where temporal and spatial variables $(t,x)\in D=\{t,x)|t\in\mathbb{R}_{+},x\in[0,L]\}$. $\rho, u$ and $p(\rho)=A\rho^{\gamma}$ are the density, velocity and pressure respectively. Here the adiabatic gas exponent $\gamma>1$ and pressure coefficient $A$ is normalized to $1$. Damping coefficient $\beta=\beta(t,x)$ is a non-positive $C^{2}$ smooth function.\\
\indent In recent years, much effort has been made on the time-periodic solutions to the viscous fluids equations and the hyperbolic conservation laws, see for example~\cite{Cai,Luo,Ma,Matsumura,G,Ohnawa,Takeno,Temple,Naoki}. However, the time-periodic solutions mentioned above are usually caused by the time-periodic external forces. As far as we know, there is little work to consider the problem with time-periodical boundary. In 2019, Yuan~\cite{Yuan} studied the existence and high-frequency limiting behavior of supersonic time-periodic solutions to the 1-D isentropic compressible Euler equations (i.e. $\beta(t,x)\equiv0$) with time-periodic inflow boundary conditions. Supersonic is an essential assumption in~\cite{Yuan}, which implies that all characteristic propagate forward in both space and time. Therefore, due to the finite propagation speed, Yuan obtained the existence of time-periodic solution after the start-up time $T_{0}>0$, which are needed for the boundary condition to effect the whole domain. From a physical point of view, the force produced by the wall of duct~\cite{Shapiro} or porous medium can be regarded as some source term added to the Euler equations. A natural question is: what kind of effects can the friction term do on the time-periodic solution$?$ \cite{Yu, Yuh} considered the existence and stability of supersonic time periodic flows for the compressible Euler equations with friction term.
Due to wave interactions between families of different directions, it is well known that the case with subsonic boundary conditions is more complicated for compressible Euler equation. Motivated by~\cite{Qu}, in which Qu considered the time-periodic solutions triggered by a kind of dissipative time-periodic boundary condition for the general quasilinear hyperbolic systems, we consider the subsonic time-periodical solutions of isentropic Euler equation with linear damping and the damping coefficient $\beta(t,x)$ satisfy the following hypothesis: \\
\emph{(H):~there exist two constants $\beta_{*}<0$ and $T_{*}>0$ such that
\begin{align}
&\beta_{*}\leq\beta(t,x)\leq0,\label{a2}\\
&|\partial_{t}\beta(t,x)|,|\partial_{x}\beta(t,x)|\leq\epsilon,\label{a3}\\
&\beta(t+T_{*},x)=\beta(t,x),\label{a4}
\end{align}
where the small constant $\epsilon$ will be determined later.}

Obviously, \eqref{a2}-\eqref{a4} hold if $\beta(t,x)$ is a negative constant.\\
\indent It should be noted that: the third author Zhang~\cite{Zhang} studied the same problem with~$\beta(t,x)=\beta(t)$. However, \cite{Zhang} needs the smallness of the friction coefficient $\beta(t)$ and its integral needs to be zero in one cycle, i.e. $$\displaystyle\int_t^{t+T_{*}} \beta(s)ds=0,$$
 which can not be satisfied for general constants. We remove these two restrictions and include the constant friction coefficient case in this paper. The main idea used in this paper is that we select a special linearized iteration scheme~\eqref{c3} and~\eqref{c5}. The advantage of this modified iteration scheme is to make better use of the nonstrictly diagonally dominant structure brought by the linear damping. While the iteration scheme used in~\cite{Zhang} can only utilize the smallness of the source term. \cite{ZhangX} also studied this problem with nonlinear damping. \\
\indent The rest of this paper is organized as follows. In~\secref{s2}, we first introduce the Riemann invariants of the homogeneous compressible Euler equations and give the main results:~\theref{t1}--\theref{t5}. In~\secref{s3}, we use the linearized iteration method to prove ~\theref{t1}. In~\secref{s4}, we use the inductive method to prove~\theref{t2}. In~\secref{s5} and~\secref{s6}, we consider the higher regularity and stability of the time-periodic solution.

\section{Prelimilaries and Main Results}\label{s2}
\indent\indent In this section, we make some transformations on equations~\eqref{a1} and give the main results of this paper.\\
\indent Firstly, it is easy to check that the two eigenvalues of system ~\eqref{a1} are
$$
\lambda_{1}=u-c,\quad\lambda_{2}=u+c.
$$
Then there holds
\begin{align*}
\lambda_{1}(\underline{\rho},0)<0<\lambda_{2}(\underline{\rho},0)
\end{align*}
and
\begin{align}
\lambda_{1}(\rho,u)<0<\lambda_{2}(\rho,u),\quad (\rho,u)\in\Im\label{b1}
\end{align}
for any positive constant $\underline{\rho}>0$ and a small neighborhood $\Im$ of $(\underline{\rho},0)$.\\
\indent Using the Riemann invariants $m$ and $n$ defined by
\begin{equation}\label{b2}
m=\frac{1}{2}(u-\frac{2}{\gamma-1}c),\quad
n=\frac{1}{2}(u+\frac{2}{\gamma-1}c),
\end{equation}
the system~\eqref{a1} becomes the following diagonal form
\begin{align}\label{b3}
\left\{
\begin{aligned}
m_{t}+\lambda_{1}(m,n)m_{x}=\frac{\beta}{2}(m+n),\\
n_{t}+\lambda_{2}(m,n)n_{x}=\frac{\beta}{2}(m+n),
\end{aligned}\right.
\end{align}
where $$c=\sqrt{\frac{\partial p}{\partial\rho}}=\sqrt{\gamma}\rho^{\frac{\gamma-1}{2}},$$
$$\lambda_{1}(m,n)=\frac{\gamma+1}{2}m+\frac{3-\gamma}{2}n,\quad\lambda_{2}(m,n)=\frac{3-\gamma}{2}m+\frac{\gamma+1}{2}n.$$
We impose the system~\eqref{a1} or~\eqref{b3} with the following initial data and boundary conditions
\begin{align}
t=0:\quad m(0,x)&=m_{0}(x),~~n(0,x)=n_{0}(x),\label{b4}\\
x=0:\quad n(t,0)&=n_{b}(t)+\kappa_{2}(m(t,0)-\underline{m}),\label{b5}\\
x=L:\quad m(t,L)&=m_{b}(t)+\kappa_{1}(n(t,L)-\underline{n}),\label{b6}
\end{align}
where the constants $|\kappa_{1}|<1$, $|\kappa_{2}|<1$ and $-\underline{m}=\underline{n}=\frac{\sqrt{\gamma}}{\gamma-1}\underline{\rho}^{\frac{\gamma-1}{2}}$. Moreover, in~\eqref{b5}-\eqref{b6}, we assume $m_{b}(t),n_{b}(t)$ are two periodic functions with the period $T_{*}>0$, namely,
$$ m_{b}(t+T_{*})=m_{b}(t),\quad n_{b}(t+T_{*})=n_{b}(t).$$
\eqref{b5}-\eqref{b6} is a kind of dissipative boundary condition in the sense of~\cite{Li}.\\
\indent Denote the perturbation variable
$$ \phi(t,x)=(\phi_{1}(t,x),\phi_{2}(t,x))^{\top}\mathop{=}\limits^{def.}(m(t,x)-\underline{m},n(t,x)-\underline{n})^{\top}$$
and let $\Phi=(\underline{m},\underline{n})^{\top}$, then equations~\eqref{b3} can be rewritten as
\begin{align}\label{b7}
\left\{
\begin{aligned}
\partial_{t}\phi_{1}+\lambda_{1}(\phi+\Phi)\partial_{x}\phi_{1}=\frac{\beta}{2}(\phi_{1}+\phi_{2}),\\
\partial_{t}\phi_{2}+\lambda_{2}(\phi+\Phi)\partial_{x}\phi_{2}=\frac{\beta}{2}(\phi_{1}+\phi_{2}),
\end{aligned}\right.
\end{align}
where $\lambda_{1}(\Phi)=-\underline{c}=-\sqrt{\gamma}\underline{\rho}^{\frac{\gamma-1}{2}}, \lambda_{2}(\Phi)=\underline{c}=\sqrt{\gamma}\underline{\rho}^{\frac{\gamma-1}{2}}.$\\
\indent Correspondingly, the initial data and boundary conditions~\eqref{b4}-\eqref{b6} become
\begin{align}
t=0:~~\phi(0,x)&=\phi_{0}(x)=(\phi_{1_{0}}(x),\phi_{2_{0}}(x))^{\top}\notag\\
&\hspace{1.3cm}=(m_{0}(x)-\underline{m},n_{0}(x)-\underline{n})^{\top},\label{b8}\\
x=0:~~\phi_{2}(t,0)&=\phi_{2_{b}}(t)+\kappa_{2}\phi_{1}(t,0),\quad t\geq0,\label{b9}\\
x=L:~~\phi_{1}(t,L)&=\phi_{1_{b}}(t)+\kappa_{1}\phi_{2}(t,L),\quad t\geq0,\label{b10}
\end{align}
where $\phi_{1_{b}}(t)=m_{b}(t)-\underline{m},\phi_{2_{b}}(t)=n_{b}(t)-\underline{n}$ satisfying
$$\phi_{i_{b}}(t+T_{*})=\phi_{i_{b}}(t),\quad t>0, i=1,2.$$
Obviously, we have the following two facts:\\
$\mathbf{1}:~$By~\eqref{b1}, we get
\begin{align*}
\lambda_{1}(\phi+\Phi)<0<\lambda_{2}(\phi+\Phi),\quad \forall \phi\in \Psi,
\end{align*}
where $\Psi$ is a small neighborhood of $O=(0,0)^{\top}$ corresponding to $\Im$.\\
$\mathbf{2}:~$Define $ \nu_{i}(\phi+\Phi)=\lambda_{i}^{-1}(\phi+\Phi)(i=1,2)$, then there exists a positive constant $A_{0}$, such that
\begin{align}\label{b13}
\mathop{\max}\limits_{i=1,2}\mathop{\sup}\limits_{\phi\in \Psi}|\nu_{i}(\phi+\Phi)|\leq A_{0}.
\end{align}
Next, we give the main results in the following theorems.
\begin{theorem}\label{t1}
(Existence of time-periodic solutions) There exists a small enough constant $\epsilon_{1}>0$ and a constant $C_{E}>0$ such that for any given $0<\epsilon<\epsilon_{1}$, and any given $C^{1}$ smooth functions $\phi_{i_{b}}(t)(i=1,2)$ satisfying
\begin{align}
& \phi_{i_{b}}(t+T_{*})=\phi_{i_{b}}(t),\quad t>0, i=1,2,\label{b11}\\
& \|\phi_{i_{b}}(t)\|_{C^{1}(\mathbb{R}_{+})}\leq\epsilon,\label{b12}
\end{align}
there exists a $C^{1}$ smooth function $\phi_{0}(x)$ satisfying
\begin{align}
\|\phi_{0}\|_{C^{1}([0,L])}\leq C_{E}\epsilon,\label{b14}
\end{align}
such that the initial-boundary value problem~\eqref{b7}-\eqref{b10} admits a $C^{1}$ time-periodic solution $\phi=\phi^{(T_{*})}(t,x)$ on $D=\{(t,x)|t\in \mathbb{R}_{+},x\in [0,L]\}$ which satisfies
\begin{align}
\phi^{(T_{*})}(t+T_{*},x)&=\phi^{(T_{*})}(t,x),\quad \forall(t,x)\in D\label{b15}
\end{align}
and
\begin{align}
\|\phi^{(T_{*})}\|_{C^{1}(D)}&\leq C_{E}\epsilon.\label{b16}
\end{align}
\end{theorem}
\begin{theorem}\label{t2}
(Stability of time-periodic solutions) There exists a small constant $\epsilon_{2}\in(0,\epsilon_{1})$ and a constant $C_{S}>0$, such that for any given $\epsilon\in(0,\epsilon_{2})$ and any given $C^{1}$ smooth functions $\phi_{0}=\phi_{0}(x)$ and $\phi_{i_{b}}(t)(i=1,2)$ satisfying \eqref{b14} and \eqref{b11}-\eqref{b12} with certain compatibilities, the initial-boundary value problem \eqref{b7}-\eqref{b10} has a unique global $C^{1}$ classical solution $\phi=\phi(t,x)$ on $D=\{(t,x)|t\in\mathbb{R}_{+},x\in[0,L]\}$ satisfying
\begin{align}
\|\phi(t,\cdot)-\phi^{(T_{*})}(t,\cdot)\|_{C^{0}}\leq C_{S}\epsilon\xi^{[t/T_{0}]},\quad\forall t\geq0,\label{b17}
\end{align}
where $\phi^{(T_{*})}$, depending on $\phi_{i_{b}}(t)(i=1,2)$, is the time-periodic solution given through \theref{t1}, $\xi\in(0,1)$ is a constant and  $T_{0}=\mathop{\max}\limits_{i=1,2}\mathop{\sup}\limits_{\phi\in \Phi}\frac{L}{|\lambda_{i}(\phi+\underline{\phi})|}$.
\end{theorem}
\indent The uniqueness of the time-periodic solution is a direct consequence from~\theref{t2} by taking $t\rightarrow+\infty$.
\begin{corollary}\label{t3}
(Uniqueness of the time-periodic solution) There exists a constant $\epsilon_{3}\in(0,\epsilon_{2})$, such that for any given $\epsilon\in(0,\epsilon_{3})$ and any given $C^{1}$ smooth functions $\phi_{i_{b}}(t)(i=1,2)$ satisfying~\eqref{b11}-\eqref{b12}, the corresponding time-periodic solution $\phi=\phi^{(T_{*})}(t,x)$ obtained in~\theref{t1} is unique.
\end{corollary}
\begin{theorem}\label{t4}
(Regularity of the time-periodic solution) If $\phi_{i_{b}}(t)(i=1,2)$ satisfy~\eqref{b11}-\eqref{b12} and possess further $W^{2,\infty}$ regularity with
\begin{align}\label{b18}
\mathop{\max}\limits_{i=1,2}\|m_{i_{b}}''(t)\|_{L^{\infty}}\leq M_{2}<+\infty,
\end{align}
then there exist constants $C_{R}>0$ and $\epsilon_{4}\in(0,\epsilon_{1})$, such that for any given $\epsilon\in(0,\epsilon_{4})$, the time-periodic solution $\phi=\phi^{(T_{*})}(t,x)$ provided by~\theref{t1} is also a $W^{2,\infty}$ function with
\begin{align}\label{b19}
\max\{\|\partial_{t}^{2}\phi^{(T_{*})}\|_{L^{\infty}},\|\partial_{t}\partial_{x}\phi^{(T_{*})}\|_{L^{\infty}},\|\partial_{x}^{2}\phi^{(T_{*})}\|_{L^{\infty}}\}
\leq(1+A_{0})^{2}C_{R}<+\infty,
\end{align}
where $A_{0}$ is defined in~\eqref{b13}.
\end{theorem}
\begin{theorem}\label{t5}
(Stabilization around the time-periodic solution) Assume that~\eqref{b18} holds, then there exist constants $C_{S}^{*}>0$ and $\epsilon_{5}\in(0,\min\{\epsilon_{2},\epsilon_{4}\})$, such that for any given $\epsilon\in(0,\epsilon_{5})$, we have not only the $C^{0}$ convergence result~\eqref{b17} in~\theref{t2}, but also the $C^{1}$ exponential convergence as
\begin{align}
\max\{\|\partial_{t}\phi(t,\cdot)-\partial_{t}\phi^{(T_{*})}(t,\cdot)\|_{C^{0}},\|\partial_{x}\phi(t,\cdot)-\partial_{x}\phi^{(T_{*})}(t,\cdot)\|_{C^{0}}\}
&\leq A_{0}C_{S}^{*}\epsilon\xi^{[t/T_{0}]},\notag\\
&\forall t\geq0.\label{b20}
\end{align}
\end{theorem}
\section{Existence of Time-periodic Solutions}\label{s3}
\indent\indent In this section, we give the proof of~\theref{t1} by applying the linearized iteration method.\\
\indent Firstly, multiplying $\nu_{i}(\phi+\Phi)=\lambda_{i}^{-1}(\phi+\Phi)(i=1,2)$ on both sides of the $i$-th equation of~\eqref{b7} for $i=1,2$ and swapping the positions of $t$ and $x$, we have
\begin{align}
\partial_{x}\phi_{1}+\nu_{1}(\phi+\Phi)\partial_{t}\phi_{1}=&\frac{\beta}{2}\nu_{1}(\Phi)\phi_{1}+\frac{\beta}{2}\Big(\nu_{1}(\phi+\Phi)-\nu_{1}(\Phi)\Big)\phi_{1}\notag\\
&+\frac{\beta}{2}\nu_{1}(\phi+\Phi)\phi_{2},\label{c1}\\
\partial_{x}\phi_{2}+\nu_{2}(\phi+\Phi)\partial_{t}\phi_{2}=&\frac{\beta}{2}\nu_{2}(\Phi)\phi_{2}+\frac{\beta}{2}\Big(\nu_{2}(\phi+\Phi)-\nu_{2}(\Phi)\Big)\phi_{2}\notag\\
&+\frac{\beta}{2}\nu_{2}(\phi+\Phi)\phi_{1}.\label{c2}
\end{align}
Then using~\eqref{c1}-\eqref{c2} and~\eqref{b9}-\eqref{b10}, we establish the following "initial"-value problem of linearized system
\begin{align}
\partial_{x}\phi_{1}^{(l)}+\nu_{1}(\phi^{(l-1)}+\Phi)\partial_{t}\phi_{1}^{(l)}=&\frac{\beta}{2}\nu_{1}(\Phi)\phi_{1}^{(l)}+\frac{\beta}{2}\Big(\nu_{1}(\phi^{(l-1)}+\Phi)
-\nu_{1}(\Phi)\Big)\phi_{1}^{(l-1)}\notag\\
&+\frac{\beta}{2}\nu_{1}(\phi^{(l-1)}+\Phi)\phi_{2}^{(l-1)},\label{c3}
\end{align}
\begin{align}
x=L:~~~~\phi_{1}^{(l)}(t,L)&=\phi_{1_{br}}(t)+\kappa_{1}\phi_{2_{r}}^{(l-1)}(t,L)
\label{c4}
\end{align}
and
\begin{align}
\partial_{x}\phi_{2}^{(l)}+\nu_{2}(\phi^{(l-1)}+\Phi)\partial_{t}\phi_{2}^{(l)}=&\frac{\beta}{2}\nu_{2}(\Phi)\phi_{2}^{(l)}+\frac{\beta}{2}\Big(\nu_{2}(\phi^{(l-1)}+\Phi)
-\nu_{2}(\Phi)\Big)\phi_{2}^{(l-1)}\notag\\
&+\frac{\beta}{2}\nu_{2}(\phi^{(l-1)}+\Phi)\phi_{1}^{(l-1)},\label{c5}
\end{align}
\begin{align}
x=0:~~~~\phi_{2}^{(l)}(t,0)&=\phi_{2_{bl}}(t)+\kappa_{2}\phi_{1_{l}}^{(l-1)}(t,0),
\label{c6}
\end{align}
where
\begin{align*}
\phi_{1_{br}}(t)=\left\{
\begin{aligned}
\phi_{1_{b}}(t),\quad t\geq0,\\
\phi_{1_{b'}}(t),\quad t<0,
\end{aligned}\right.\quad
\phi_{2_{bl}}(t)=\left\{
\begin{aligned}
\phi_{2_{b}}(t),\quad t\geq0,\\
\phi_{2_{b'}}(t),\quad t<0
\end{aligned}\right.
\end{align*}
are the time-periodic extensions of $\phi_{i_{b}}(t)(i=1,2)$.\\
\indent Here we select a special linearized iteration scheme~\eqref{c3} and~\eqref{c5} to make better use of the nonstrictly diagonally dominant structure brought by damping.\\
\indent For system~\eqref{c3}-\eqref{c6}, we start to iterate from
\begin{align}
\phi^{(0)}(t,x)=0 \label{c7}
\end{align}
and prove the following a priori estimates.
\begin{proposition}\label{p1}
There exists a small enough constant $\epsilon_{1}>0$ and constants $M_{1}>0$, $C_{1}>0$ and $\theta\in(0,1)$, such that for any given $\epsilon\in(0,\epsilon_{1})$, the following estimates hold:
\begin{align}
&\phi^{(l)}(t+T_{*},x)=\phi^{(l)}(t,x),\quad\forall(t,x)\in \mathbb{R}\times[0,L],l\in \mathbb{N}_{+},\label{c8}\\
&\|\phi^{(l)}\|_{C^{1}}\leq (C_{1}+M_{1})\epsilon,\quad \forall l\in \mathbb{N}_{+}, \label{c9}\\
&\|\phi^{(l)}-\phi^{(l-1)}\|_{C^{0}}\leq M_{1}\epsilon\theta^{l}, \quad \forall l\in \mathbb{N}_{+},\label{c10}\\
&\mathop{\max}\limits_{i=1,2}\{\varpi(\delta|\partial_{t}\phi_{i}^{(l)})+\varpi(\delta|\partial_{x}\phi_{i}^{(l)})\}\leq(\frac{1}{3}+\frac{1}{2}A_{0})\Omega(\delta), \quad \forall l\in \mathbb{N}_{+},\label{c11}
\end{align}
where
\begin{align}
\|\phi^{(l)}\|_{C^{1}(D)}&\mathop{=}\limits^{def.}\mathop{\max}\limits_{i=1,2}\{\|\phi_{i}^{(l)}\|_{C^{0}(D)},\|\partial_{t}\phi_{i}^{(l)}\|_{C^{0}(D)}
 ,\|\partial_{x}\phi_{i}^{(l)}\|_{C^{0}(D)}\},\notag\\
\varpi(\delta|h)&=\mathop{\sup}\limits_{\substack{|t_{1}-t_{2}|\leq\delta\\|x_{1}-x_{2}|\leq\delta}}|h(t_{1},x_{1})-h(t_{2},x_{2})|,\notag
\end{align}
and $\Omega(\delta)$ is a continuous function of $\delta\in(0,1)$ which is independent of $l$ and satisfies
$$
\mathop{\lim}\limits_{\delta\rightarrow0^{+}}\Omega(\delta)=0.
$$
\end{proposition}
\begin{proof}
We prove the a priori estimates~\eqref{c8}-\eqref{c11} inductively, i.e., for each $l\in \mathbb{N}_{+}$, we show
\begin{align}
&\phi_{i}^{(l)}(t+T_{*},x)=\phi_{i}^{(l)}(t,x), \quad\forall(t,x)\in \mathbb{R}\times[0,L], \forall i=1,2,\label{c12}\\
&\mathop{\max}\limits_{i=1,2}\{\|\phi_{i}^{(l)}\|_{C^{0}},\|\partial_{t}\phi_{i}^{(l)}\|_{C^{0}}\}\leq M_{1}\epsilon,\label{c13}\\
&\mathop{\max}\limits_{i=1,2}\{\|\partial_{x}\phi_{i}^{(l)}\|_{C^{0}}\}\leq C_{1}\epsilon,\label{c14}\\
&\mathop{\max}\limits_{i=1,2}\|\phi_{i}^{(l)}-\phi_{i}^{(l-1)}\|_{C^{0}}\leq M_{1}\epsilon\theta^{l},\label{c15}\\
&\mathop{\max}\limits_{i=1,2}\varpi(\delta|\partial_{t}\phi_{i}^{(l)}(\cdot,x))\leq\frac{1}{8[A_{0}+1]}\Omega(\delta),\quad \forall x\in[0,L]\label{c16}
\end{align}
and
\begin{align}
\mathop{\max}\limits_{i=1,2}\{\varpi(\delta|\partial_{t}\phi_{i}^{(l)})+\varpi(\delta|\partial_{x}\phi_{i}^{(l)})\}\leq(\frac{1}{3}+\frac{1}{2}A_{0})\Omega(\delta)\label{c17}
\end{align}
under the following hypothesis
\begin{align}
&\phi_{i}^{(l-1)}(t+T_{*},x)=\phi_{i}^{(l-1)}(t,x), \quad\forall(t,x)\in \mathbb{R}\times[0,L], \forall i=1,2,\label{c18}\\
&\mathop{\max}\limits_{i=1,2}\{\|\phi_{i}^{(l-1)}\|_{C^{0}},\|\partial_{t}\phi_{i}^{(l-1)}\|_{C^{0}}\}\leq M_{1}\epsilon,\label{c19}\\
&\mathop{\max}\limits_{i=1,2}\{\|\partial_{x}\phi_{i}^{(l-1)}\|_{C^{0}}\}\leq C_{1}\epsilon,\label{c20}\\
&\mathop{\max}\limits_{i=1,2}\|\phi_{i}^{(l-1)}-\phi_{i}^{(l-2)}\|_{C^{0}}\leq M_{1}\epsilon\theta^{l},\quad \forall l\geq2,\label{c21}\\
&\mathop{\max}\limits_{i=1,2}\varpi(\delta|\partial_{t}\phi_{i}^{(l-1)}(\cdot,x))\leq\frac{1}{8[A_{0}+1]}\Omega(\delta),\quad \forall x\in[0,L]\label{c22}
\end{align}
and
\begin{align}
\mathop{\max}\limits_{i=1,2}\{\varpi(\delta|\partial_{t}\phi_{i}^{(l-1)})+\varpi(\delta|\partial_{x}\phi_{i}^{(l-1)})\}\leq(\frac{1}{3}+\frac{1}{2}A_{0})\Omega(\delta),\label{c23}
\end{align}
where $[A_{0}+1]$ represents the integer part of $A_{0}+1$ and
$$\varpi(\delta|h(\cdot,x))=\mathop{\max}\limits_{|t_{1}-t_{2}|\leq\delta}|h(t_{1},x)-h(t_{2},x)|.$$
In the remaining parts of this section, we will use several steps to check the above estimates~\eqref{c12}-\eqref{c17} one by one.\\
$\mathbf{Step~1}:~$ Transformations of the principal equations.\\
Define the characteristic curve $t=t_{i}^{(l)}(x;t_{0},x_{0})(i=1,2)$ as the following form:
\begin{align}
\left\{
\begin{aligned}
&\frac{dt_{i}^{(l)}}{dx}(x;t_{0},x_{0})=\nu_{i}(\phi^{(l-1)}+\Phi)(t_{i}^{(l)}(x;t_{0},x_{0}),x),\\
&t_{i}^{(l)}(x_{0};t_{0},x_{0})=t_{0}.
\end{aligned}\right.\label{c24}
\end{align}
Denote
$$ F_{1}(t,x)=e^{\int_{x}^{L}\frac{\beta(t,s)}{2}\nu_{1}(\Phi)ds},~~F_{2}(t,x)=e^{-\int_{0}^{x}\frac{\beta(t,s)}{2}\nu_{2}(\Phi)ds}.$$
Noticing $\beta(t,x)\leq0,\nu_{1}(\Phi)<0,\nu_{2}(\Phi)>0$ and $x\in [0,L]$, we have
\begin{align}
&F_{1}(t,x),F_{2}(t,x)\geq1,\label{c25}\\
&\frac{\partial F_{1}(t,x)}{\partial x}=-\frac{\beta(t,x)}{2}\nu_{1}(\Phi)F_{1}(t,x)<0,\label{c26}\\
&\frac{\partial F_{1}(t,x)}{\partial t}=\int_{x}^{L}\frac{\partial_{t}\beta(t,s)}{2}\nu_{1}(\Phi)ds F_{1}(t,x),\notag\\
&\frac{\partial F_{2}(t,x)}{\partial x}=-\frac{\beta(t,x)}{2}\nu_{2}(\Phi)F_{2}(t,x)>0,\label{c27}\\
&\frac{\partial F_{2}(t,x)}{\partial t}=-\int_{0}^{x}\frac{\partial_{t}\beta(t,s)}{2}\nu_{2}(\Phi)ds F_{2}(t,x).\notag
\end{align}
Furthermore, by~\eqref{a2} and~\eqref{b13}, we get
\begin{align}
&1\leq F_{1}(t,x)\leq e^{-\frac{\beta_{*}}{2}A_{0}L}\mathop{=}\limits^{def.}M_{0}>1,\label{c28}\\
&1\leq F_{2}(t,x)\leq M_{0}.\label{c29}
\end{align}
Let $M_{1}=\frac{100}{1-\kappa}$ with $\kappa=\max\{|\kappa_{1}|,|\kappa_{2}|\}<1$,then
\begin{align}
M_{1}\geq|\kappa_{1}|M_{1}+100,\quad M_{1}\geq|\kappa_{2}|M_{1}+100.\label{c30}
\end{align}
Now, we turn problem~\eqref{c3}-\eqref{c6} into the system of $F_{i}(t,x)\phi_{i}^{(l)}$ as follows:
\begin{align}
&\partial_{x}\Big(F_{1}\phi_{1}^{(l)}\Big)+\nu_{1}(\phi^{(l-1)}+\Phi)\partial_{t}\Big(F_{1}\phi_{1}^{(l)}\Big)\notag\\
=&\frac{\beta}{2}F_{1}\Big(\nu_{1}(\phi^{(l-1)}+\Phi)-\nu_{1}(\Phi)\Big)(\phi_{1}^{(l-1)}+\phi_{2}^{(l-1)})\notag\\
&+\frac{\beta}{2}F_{1}\nu_{1}(\Phi)\phi_{2}^{(l-1)}+\nu_{1}(\phi^{(l-1)}+\Phi)F_{1}\phi_{1}^{(l)}\int_{x}^{L}\frac{\partial_{t}\beta(t,s)}{2}\nu_{1}(\Phi)ds,\label{c31}
\end{align}
\begin{align}
x=L:~~~~F_{1}(t,L)\phi_{1}^{(l)}(t,L)=\phi_{1}^{(l)}(t,L)=\phi_{1_{br}}(t)+\kappa_{1}\phi_{2_{r}}^{(l-1)}(t,L),\label{c32}
\end{align}
\begin{align}
&\partial_{x}\Big(F_{2}\phi_{2}^{(l)}\Big)+\nu_{2}(\phi^{(l-1)}+\Phi)\partial_{t}\Big(F_{2}\phi_{2}^{(l)}\Big)\notag\\
=&\frac{\beta}{2}F_{2}\Big(\nu_{2}(\phi^{(l-1)}+\Phi)-\nu_{2}(\Phi)\Big)(\phi_{1}^{(l-1)}+\phi_{2}^{(l-1)})\notag\\
&+\frac{\beta}{2}F_{2}\nu_{2}(\Phi)\phi_{1}^{(l-1)}-\nu_{2}(\phi^{(l-1)}+\Phi)F_{2}\phi_{2}^{(l)}\int_{0}^{x}\frac{\partial_{t}\beta(t,s)}{2}\nu_{2}(\Phi)ds,\label{c33}
\end{align}
\begin{align}
x=0:~~~~F_{2}(t,0)\phi_{2}^{(l)}(t,0)=\phi_{2}^{(l)}(t,0)=\phi_{2_{bl}}(t)+\kappa_{2}\phi_{1_{l}}^{(l-1)}(t,0).\label{c34}
\end{align}
By~\eqref{a4}, \eqref{b11} and~\eqref{c18}, it is easy to know that if $F_{i}(t,x)\phi_{i}^{(l)}(t,x)(i=1,2)$ solves problem~\eqref{c31}-\eqref{c34}, so does $F_{i}(t+T_{*},x)\phi_{i}^{(l)}(t+T_{*},x)(i=1,2).$ Then we get $F_{i}(t+T_{*},x)\phi_{i}^{(l)}(t+T_{*},x)=F_{i}(t,x)\phi_{i}^{(l)}(t,x)$ by the uniqueness of this linear system. Thus~\eqref{c12} is proved.\\
$\mathbf{Step~2}:~$ The proof of~\eqref{c13} and~\eqref{c14}.\\
\indent Next, we prove the $C^{0}$ estimate for $\phi_{i}^{(l)}(i=1,2)$. By the aid of conditions~\eqref{b12}, \eqref{c4}, \eqref{c6}, \eqref{c30}, we get
\begin{align}
&\|\phi_{1}^{(l)}(\cdot,L)\|_{C^{0}}\leq\epsilon+|\kappa_{1}|M_{1}\epsilon\leq M_{1}\epsilon-99\epsilon,\label{c35}\\
&\|\phi_{2}^{(l)}(\cdot,0)\|_{C^{0}}\leq M_{1}\epsilon-99\epsilon.\label{c36}
\end{align}
Then we integrate~\eqref{c31} along the $1$st characteristic curve $t=t_{1}^{(l)}(y;t,x)$ from $L$ to $x$ to get
\begin{align*}
&F_{1}(t,x)\phi_{1}^{(l)}(t,x)-F_{1}(t,L)\phi_{1}^{(l)}(t_{1}^{(l)}(L;t,x),L)\notag\\
=&\int^{x}_{L}\Big(\frac{\beta(\tau,y)}{2}F_{1}(\tau,y)\big(\nu_{1}(\phi^{(l-1)}+\Phi)-\nu_{1}(\Phi)\big)(\phi_{1}^{(l-1)}+\phi_{2}^{(l-1)})\notag\\
&+\nu_{1}(\phi^{(l-1)}+\Phi)F_{1}(\tau,y)\phi_{1}^{(l)}\int_{y}^{L}\frac{\partial_{t}\beta(\tau,s)}{2}\nu_{1}(\Phi)ds\Big)\big|_{\tau=t_{1}^{(l)}(y;t,x)}dy\notag\\
&+\int^{x}_{L}-\big(\frac{\partial}{\partial y}F_{1}(t_{1}^{(l)}(y;t,x),y)\big)\phi_{2}^{(l-1)}(t_{1}^{(l)}(y;t,x),y)dy,
\end{align*}
where $F_{1}(t,L)=1$.\\
\indent Using~\eqref{a2}-\eqref{a3}, \eqref{b13}, \eqref{c19} and~\eqref{c35}, we get
\begin{align}
\|\phi_{1}^{(l)}\|_{C^{0}}&\leq\frac{M_{1}\epsilon-99\epsilon}{F_{1}(t,x)}+C\epsilon^{2}+\frac{F_{1}(t,x)-1}{F_{1}(t,x)}M_{1}\epsilon\notag\\
&\leq\frac{M_{1}\epsilon-99\epsilon}{F_{1}(t,x)}+\frac{\epsilon}{F_{1}(t,x)}+\frac{F_{1}(t,x)-1}{F_{1}(t,x)}M_{1}\epsilon\notag\\
&=M_{1}\epsilon-\frac{98\epsilon}{F_{1}(t,x)}<M_{1}\epsilon,\label{c37}
\end{align}
where in the above second inequality, we used $C\epsilon^{2}<\frac{\epsilon}{F_{1}(t,x)}$ for a suitably small $\epsilon$.
Unless specified, in this paper $C$ denotes a generic constant.\\
\indent Similarly, we can get
\begin{align}
\|\phi_{2}^{(l)}\|_{C^{0}}<M_{1}\epsilon. \label{c38}
\end{align}
Now, we consider the temporal derivative estimates.\\
\indent Let
$$z_{i}^{(l)}=\partial_{t}\phi_{i}^{(l)},\quad i=1,2, \forall l\in \mathbb{N}.$$
Differentiating equations~\eqref{c3}, \eqref{c5}
with respect to $t$, we get
\begin{align}
&\partial_{x}z_{1}^{(l)}+\nu_{1}(\phi^{(l-1)}+\Phi)\partial_{t}z_{1}^{(l)}\notag\\
=&\frac{\beta}{2}\nu_{1}(\Phi)z_{1}^{(l)}+\frac{\partial_{t}\beta}{2}\nu_{1}(\Phi)\phi_{1}^{(l)}+\frac{\beta}{2}\nu_{1}(\Phi)z_{2}^{(l-1)}+\frac{\partial_{t}\beta}{2}\nu_{1}(\Phi)\phi_{2}^{(l-1)}\notag\\
&+\frac{\beta}{2}\Big(\nu_{1}(\phi^{(l-1)}+\Phi)-\nu_{1}(\Phi)\Big)(z_{1}^{(l-1)}+z_{2}^{(l-1)})\notag\\
&+\frac{\partial_{t}\beta}{2}\Big(\nu_{1}(\phi^{(l-1)}+\Phi)-\nu_{1}(\Phi)\Big)(\phi_{1}^{(l-1)}+\phi_{2}^{(l-1)})\notag\\
&+\frac{\beta}{2}\sum_{j=1}^{2}\frac{\partial\nu_{1}(\phi^{(l-1)}+\Phi)}{\partial\phi_{j}}(z_{j}^{(l-1)}\phi_{1}^{(l-1)}+z_{j}^{(l-1)}\phi_{2}^{(l-1)})\notag\\
&-\sum_{j=1}^{2}\frac{\partial\nu_{1}(\phi^{(l-1)}+\Phi)}{\partial\phi_{j}}z_{j}^{(l-1)}z_{1}^{(l)},\label{c39}
\end{align}
and
\begin{align}
&\partial_{x}z_{2}^{(l)}+\nu_{2}(\phi^{(l-1)}+\Phi)\partial_{t}z_{2}^{(l)}\notag\\
=&\frac{\beta}{2}\nu_{2}(\Phi)z_{2}^{(l)}+\frac{\partial_{t}\beta}{2}\nu_{2}(\Phi)\phi_{2}^{(l)}+\frac{\beta}{2}\nu_{2}(\Phi)z_{1}^{(l-1)}+\frac{\partial_{t}\beta}{2}\nu_{2}(\Phi)\phi_{1}^{(l-1)}\notag\\
&+\frac{\beta}{2}\Big(\nu_{2}(\phi^{(l-1)}+\Phi)-\nu_{2}(\Phi)\Big)(z_{1}^{(l-1)}+z_{2}^{(l-1)})\notag\\
&+\frac{\partial_{t}\beta}{2}\Big(\nu_{2}(\phi^{(l-1)}+\Phi)-\nu_{2}(\Phi)\Big)(\phi_{1}^{(l-1)}+\phi_{2}^{(l-1)})\notag\\
&+\frac{\beta}{2}\sum_{j=1}^{2}\frac{\partial\nu_{2}(\phi^{(l-1)}+\Phi)}{\partial\phi_{j}}(z_{j}^{(l-1)}\phi_{1}^{(l-1)}+z_{j}^{(l-1)}\phi_{2}^{(l-1)})\notag\\
&-\sum_{j=1}^{2}\frac{\partial\nu_{2}(\phi^{(l-1)}+\Phi)}{\partial\phi_{j}}z_{j}^{(l-1)}z_{2}^{(l)}.\label{c41}
\end{align}
Correspondingly, at the boundary $x=0$ and $x=L$, we have
\begin{align}
x=L:~~~~z_{1}^{(l)}(t,L)=\phi'_{1_{br}}(t)+\kappa_{1}z_{2_{r}}^{(l-1)}(t,L),\quad t\in \mathbb{R},\label{c40}
\end{align}
and
\begin{align}
x=0:~~~~z_{2}^{(l)}(t,0)=\phi'_{2_{bl}}(t)+\kappa_{2}z_{1_{l}}^{(l-1)}(t,0),\quad t\in \mathbb{R}.\label{c42}
\end{align}
By~\eqref{b12} and~\eqref{c19}, we get
\begin{align}
&\|z_{1}^{(l)}(\cdot,L)\|_{C^{0}}\leq\epsilon+|\kappa_{1}|M_{1}\epsilon\leq M_{1}\epsilon-99\epsilon,\label{c43}\\
&\|z_{2}^{(l)}(\cdot,0)\|_{C^{0}}\leq\epsilon+|\kappa_{2}|M_{1}\epsilon\leq M_{1}\epsilon-99\epsilon.\label{c44}
\end{align}
In order to overcome the complexity of linearized equations caused by the quasilinear nature of the original system, we introduce the method of "estimation by twice integration". We first get a rough estimate of $z_{1}^{(l)}$ as follows.\\
\indent Multiplying $\mathrm{sgn}(z_{1}^{(l)})$ on both sides of~\eqref{c39}, we get
\begin{align*}
&\partial_{x}|z_{1}^{(l)}|+\nu_{1}(\phi^{(l-1)}+\Phi)\partial_{t}|z_{1}^{(l)}|\notag\\
=&\Big(\frac{\beta}{2}\nu_{1}(\Phi)-\sum_{j=1}^{2}\frac{\partial\nu_{1}}{\partial\phi_{j}}z_{j}^{(l-1)}\Big)|z_{1}^{(l)}|+\frac{\partial_{t}\beta}{2}\mathrm{sgn}(z_{1}^{(l)})\nu_{1}(\Phi)\phi_{1}^{(l)}\notag\\
&+\frac{\beta}{2}\mathrm{sgn}(z_{1}^{(l)})\nu_{1}(\Phi)z_{2}^{(l-1)}+\frac{\partial_{t}\beta}{2}\mathrm{sgn}(z_{1}^{(l)})\nu_{1}(\Phi)\phi_{2}^{(l-1)}\notag\\
&+\frac{\beta}{2}\mathrm{sgn}(z_{1}^{(l)})\Big(\nu_{1}(\phi^{(l-1)}+\Phi)-\nu_{1}(\Phi)\Big)(z_{1}^{(l-1)}+z_{2}^{(l-1)})\notag\\
&+\frac{\partial_{t}\beta}{2}\mathrm{sgn}(z_{1}^{(l)})\Big(\nu_{1}(\phi^{(l-1)}+\Phi)-\nu_{1}(\Phi)\Big)(\phi_{1}^{(l-1)}+\phi_{2}^{(l-1)})\notag\\
&+\frac{\beta}{2}\mathrm{sgn}(z_{1}^{(l)})\sum_{j=1}^{2}\frac{\partial\nu_{1}}{\partial\phi_{j}}(z_{j}^{(l-1)}\phi_{1}^{(l-1)}+z_{j}^{(l-1)}\phi_{2}^{(l-1)}).
\end{align*}
Using the sign of $\beta(t,x)$ in~\eqref{a2} and noticing~\eqref{b13}, \eqref{c19}, it is easy to see
$$\frac{\beta}{2}\nu_{1}(\Phi)-\sum_{j=1}^{2}\frac{\partial\nu_{1}}{\partial\phi_{j}}z_{j}^{(l-1)}>0.$$
Then integrating the result along the $1$st characteristic curve $t=t_{1}^{(l)}(y;t,x)$ from $L$ to $x$ and using~\eqref{a2}-\eqref{a3}, \eqref{b13}, \eqref{c19} and~\eqref{c43}, one has
\begin{align}
|z_{1}^{(l)}(t,x)|\leq&|z_{1}^{(l)}(t_{1}^{(l)}(L;t,x),L)|+(\frac{|\beta|}{2}|\nu_{1}(\Phi)|+|\partial_{t}\beta||\nu_{1}(\Phi)|)|\int_{L}^{x}M_{1}\epsilon dy| +C\epsilon^{2}\notag\\
\leq& C_{0}\epsilon, \quad\forall(t,x)\in \mathbb{R}\times[0,L],\label{c45}
\end{align}
where $C_{0}>0$ is a constant independent of $l$.\\
\indent Next we transform equation~\eqref{c39} into the following equation of $F_{1}z_{1}^{(l)}$
\begin{align}
&\partial_{x}(F_{1}z_{1}^{(l)})+\nu_{1}(\phi^{(l-1)}+\Phi)\partial_{t}(F_{1}z_{1}^{(l)})\notag\\
=&\frac{\partial_{t}\beta}{2}F_{1}\nu_{1}(\Phi)\phi_{1}^{(l)}+\frac{\beta}{2}F_{1}\nu_{1}(\Phi)z_{2}^{(l-1)}
+\frac{\partial_{t}\beta}{2}F_{1}\nu_{1}(\Phi)\phi_{2}^{(l-1)}\notag\\
&+\frac{\beta}{2}F_{1}\Big(\nu_{1}(\phi^{(l-1)}+\Phi)-\nu_{1}(\Phi)\Big)(z_{1}^{(l-1)}+z_{2}^{(l-1)})\notag\\
&+\frac{\partial_{t}\beta}{2}F_{1}\Big(\nu_{1}(\phi^{(l-1)}+\Phi)-\nu_{1}(\Phi)\Big)(\phi_{1}^{(l-1)}+\phi_{2}^{(l-1)})\notag\\
&+\frac{\beta}{2}F_{1}\sum_{j=1}^{2}\frac{\partial\nu_{1}}{\partial\phi_{j}}(z_{j}^{(l-1)}\phi_{1}^{(l-1)}+z_{j}^{(l-1)}\phi_{2}^{(l-1)})
-F_{1}\sum_{j=1}^{2}\frac{\partial\nu_{1}}{\partial\phi_{j}}z_{j}^{(l-1)}z_{1}^{(l)}\notag\\
&+\nu_{1}(\phi^{(l-1)}+\Phi)F_{1}z_{1}^{(l)}\int_{x}^{L}\frac{\partial_{t}\beta(t,s)}{2}\nu_{1}(\Phi)ds,\label{c46}
\end{align}
then integrate it along the $1$st characteristic curve $t=t_{1}^{(l)}(y;t,x)$ from $L$ to $x$ to get
\begin{align*}
&F_{1}(t,x)z_{1}^{(l)}(t,x)-F_{1}(t_{1}^{(l)}(L;t,x),L)z_{1}^{(l)}(t_{1}^{(l)}(L;t,x),L)\\
=&\int_{L}^{x}-(\frac{\partial}{\partial y}F_{1}(t_{1}^{(l)}(y;t,x),y))z_{2}^{(l-1)}(t_{1}^{(l)}(y;t,x),y)dy\\
&+\int_{L}^{x}\Big(\frac{\partial_{t}\beta}{2}F_{1}\nu_{1}(\Phi)\phi_{1}^{(l)}
+\frac{\partial_{t}\beta}{2}F_{1}\nu_{1}(\Phi)\phi_{2}^{(l-1)}\\
&+\frac{\beta}{2}F_{1}\big(\nu_{1}(\phi^{(l-1)}+\Phi)-\nu_{1}(\Phi)\big)(z_{1}^{(l-1)}+z_{2}^{(l-1)})\\
&+\frac{\partial_{t}\beta}{2}F_{1}\big(\nu_{1}(\phi^{(l-1)}+\Phi)-\nu_{1}(\Phi)\big)(\phi_{1}^{(l-1)}+\phi_{2}^{(l-1)})\\
&+\frac{\beta}{2}F_{1}\sum_{j=1}^{2}\frac{\partial\nu_{1}}{\partial\phi_{j}}(z_{j}^{(l-1)}\phi_{1}^{(l-1)}+z_{j}^{(l-1)}\phi_{2}^{(l-1)})
-F_{1}\sum_{j=1}^{2}\frac{\partial\nu_{1}}{\partial\phi_{j}}z_{j}^{(l-1)}z_{1}^{(l)}\\
&+\nu_{1}(\phi^{(l-1)}+\Phi)F_{1}z_{1}^{(l)}\int_{y}^{L}\frac{\partial_{t}\beta(\cdot,s)}{2}\nu_{1}(\Phi)ds\Big)(t_{1}^{(l)}(y;t,x),y)dy.
\end{align*}
By~\eqref{a2}-\eqref{a3}, \eqref{b13}, \eqref{c19}, \eqref{c43} and the rough estimate~\eqref{c45}, we get
\begin{align}
|z_{1}^{(l)}(t,x)|&\leq\frac{1}{F_{1}(t,x)}(M_{1}\epsilon-99\epsilon)+\frac{F_{1}(t,x)-1}{F_{1}(t,x)}M_{1}\epsilon+C\epsilon^{2}\notag\\
&\leq M_{1}\epsilon-\frac{99}{F_{1}(t,x)}\epsilon+C\epsilon^{2}\notag\\
&<M_{1}\epsilon.\label{c47}
\end{align}
Similarly, we have
\begin{align}
|z_{2}^{(l)}(t,x)|<M_{1}\epsilon.\label{c48}
\end{align}
By applying the equations~\eqref{c3}, \eqref{c5} and noting~\eqref{a2}, \eqref{b13}, \eqref{c19}, \eqref{c37}-\eqref{c38} and~\eqref{c47}-\eqref{c48}, we gain
\begin{align}
\|\partial_{x}\phi_{i}^{(l)}\|_{C^{0}}&\leq A_{0}M_{1}\epsilon+\beta_{*}A_{0}M_{1}\epsilon+C\epsilon^{2}\notag\\
&\leq C_{1}\epsilon,\label{c49}
\end{align}
where we choose the constant $C_{1}>A_{0}M_{1}+\beta_{*}A_{0}M_{1}$, which is independent of $l$.\\
$\mathbf{Step~3}:~$ $\{\phi_{i}^{(l)}\}(i=1,2)$ is a Cauchy sequence in $C^{0}$.\\
\indent By~\eqref{c7}, we get~\eqref{c15} for $l=1$ from~\eqref{c37}-\eqref{c38} directly. Next we prove~\eqref{c15} for $l\geq2$.
Select $\theta<1$ satisfying
$$ \theta>|\kappa_{1}|,~~\theta>|\kappa_{2}|.$$
At the boundary $x=0$ and $x=L$, by~\eqref{c21}, we have
\begin{align}
\|\phi_{1}^{(l)}(t,L)-\phi_{1}^{(l-1)}(t,L)\|_{C^{0}}&\leq|\kappa_{1}|\|\phi_{2}^{(l-1)}(t,L)-\phi_{2}^{(l-2)}(t,L)\|_{C^{0}}\notag\\
&\leq |\kappa_{1}|M_{1}\epsilon\theta^{l-1},\label{c50}\\
\|\phi_{2}^{(l)}(t,0)-\phi_{2}^{(l-1)}(t,0)\|_{C^{0}}&\leq|\kappa_{2}|\|\phi_{1}^{(l-1)}(t,0)-\phi_{1}^{(l-2)}(t,0)\|_{C^{0}}\notag\\
&\leq |\kappa_{2}|M_{1}\epsilon\theta^{l-1}.\label{c51}
\end{align}
In the domain $D$, from~\eqref{c3}, we get
\begin{align*}
&\partial_{x}(\phi_{1}^{(l)}-\phi_{1}^{(l-1)})+\nu_{1}(\phi^{(l-1)}+\Phi)\partial_{t}(\phi_{1}^{(l)}-\phi_{1}^{(l-1)})\\
=&\frac{\beta}{2}\nu_{1}(\Phi)(\phi_{1}^{(l)}-\phi_{1}^{(l-1)})-\Big(\nu_{1}(\phi^{(l-1)}+\Phi)-\nu_{1}(\phi^{(l-2)}+\Phi)\Big)\partial_{t}\phi_{1}^{(l-1)}\\
&+\frac{\beta}{2}\Big(\nu_{1}(\phi^{(l-1)}+\Phi)-\nu_{1}(\Phi)\Big)(\phi_{1}^{(l-1)}-\phi_{1}^{(l-2)})\\
&+\frac{\beta}{2}\Big(\nu_{1}(\phi^{(l-1)}+\Phi)-\nu_{1}(\phi^{(l-2)}+\Phi)\Big)\phi_{1}^{(l-2)}\\
&+\frac{\beta}{2}\nu_{1}(\phi^{(l-1)}+\Phi)(\phi_{2}^{(l-1)}-\phi_{2}^{(l-2)})\\
&+\frac{\beta}{2}\Big(\nu_{1}(\phi^{(l-1)}+\Phi)-\nu_{1}(\phi^{(l-2)}+\Phi)\Big)\phi_{2}^{(l-2)},
\end{align*}
then we multiply $F_{1}(t,x)$ on both sides of the above equality to gain
\begin{align*}
&\partial_{x}\Big(F_{1}(\phi_{1}^{(l)}-\phi_{1}^{(l-1)})\Big)+\nu_{1}(\phi^{(l-1)}+\Phi)\partial_{t}\Big(F_{1}(\phi_{1}^{(l)}-\phi_{1}^{(l-1)})\Big)\\
=&-F_{1}\Big(\nu_{1}(\phi^{(l-1)}+\Phi)-\nu_{1}(\phi^{(l-2)}+\Phi)\Big)\partial_{t}\phi_{1}^{(l-1)}\\
&+\nu_{1}(\phi^{(l-1)}+\Phi)F_{1}(\phi_{1}^{(l)}-\phi_{1}^{(l-1)})\int_{x}^{L}\frac{\partial_{t}\beta(t,s)}{2}\nu_{1}(\Phi)ds\\
&+\frac{\beta}{2}F_{1}\Big(\nu_{1}(\phi^{(l-1)}+\Phi)-\nu_{1}(\Phi)\Big)(\phi_{1}^{(l-1)}-\phi_{1}^{(l-2)})\\
&+\frac{\beta}{2}F_{1}\Big(\nu_{1}(\phi^{(l-1)}+\Phi)-\nu_{1}(\phi^{(l-2)}+\Phi)\Big)\phi_{1}^{(l-2)}\\
&+\frac{\beta}{2}F_{1}\Big(\nu_{1}(\phi^{(l-1)}+\Phi)-\nu_{1}(\Phi)\Big)(\phi_{2}^{(l-1)}-\phi_{2}^{(l-2)})\\
&+\frac{\beta}{2}F_{1}\nu_{1}(\Phi)(\phi_{2}^{(l-1)}-\phi_{2}^{(l-2)})\\
&+\frac{\beta}{2}F_{1}\Big(\nu_{1}(\phi^{(l-1)}+\Phi)-\nu_{1}(\phi^{(l-2)}+\Phi)\Big)\phi_{2}^{(l-2)},
\end{align*}
and integrate it along the $1$st characteristic curve $t=t_{1}^{(l)}(y;t,x)$ from $L$ to $x$ to get
\begin{align*}
&F_{1}(t,x)(\phi_{1}^{(l)}(t,x)-\phi_{1}^{(l-1)}(t,x))\\
=&F_{1}(t,L)(\phi_{1}^{(l)}(t_{1}^{(l)}(L;t,x),L)-\phi_{1}^{(l-1)}(t_{1}^{(l)}(L;t,x),L))\\
&+\int_{L}^{x}\Big(-F_{1}\big(\nu_{1}(\phi^{(l-1)}+\Phi)-\nu_{1}(\phi^{(l-2)}+\Phi)\big)\partial_{t}\phi_{1}^{(l-1)}\\
&+\frac{\beta}{2}F_{1}\big(\nu_{1}(\phi^{(l-1)}+\Phi)-\nu_{1}(\Phi)\big)(\phi_{1}^{(l-1)}-\phi_{1}^{(l-2)})\\
&+\frac{\beta}{2}F_{1}\big(\nu_{1}(\phi^{(l-1)}+\Phi)-\nu_{1}(\phi^{(l-2)}+\Phi)\big)\phi_{1}^{(l-2)}\\
&+\frac{\beta}{2}F_{1}\big(\nu_{1}(\phi^{(l-1)}+\Phi)-\nu_{1}(\Phi)\big)(\phi_{2}^{(l-1)}-\phi_{2}^{(l-2)})\\
&+\frac{\beta}{2}F_{1}(\nu_{1}(\phi^{(l-1)}+\Phi)-\nu_{1}(\phi^{(l-2)}+\Phi))\phi_{2}^{(l-2)}\\
&+\nu_{1}(\phi^{(l-1)}+\Phi)F_{1}(\phi_{1}^{(l)}-\phi_{1}^{(l-1)})\int_{y}^{L}\frac{\partial_{t}\beta(\cdot,s)}{2}\nu_{1}(\Phi)ds\Big)(t_{1}^{(l)}(y;t,x),y)dy\\
&+\int_{L}^{x}-(\frac{\partial}{\partial y}F_{1}(\cdot,y))(\phi_{2}^{(l-1)}-\phi_{2}^{(l-2)})(t_{1}^{(l)}(y;t,x),y)dy.
\end{align*}
By~\eqref{c19}, \eqref{c21} and~\eqref{c50}, we have
\begin{align}
\|(\phi_{1}^{(l)}-\phi_{1}^{(l-1)})\|_{C^{0}}&\leq\frac{|\kappa_{1}|M_{1}\epsilon\theta^{l-1}}{F_{1}(t,x)}+\frac{F_{1}(t,x)-1}{F_{1}(t,x)}M_{1}\epsilon\theta^{l-1}
+C\epsilon M_{1}\epsilon\theta^{l-1}\notag\\
&\leq M_{1}\epsilon\theta^{l-1}-\frac{1-|\kappa_{1}|}{F_{1}(t,x)}M_{1}\epsilon\theta^{l-1}+C\epsilon M_{1}\epsilon\theta^{l-1}\notag\\
&\leq M_{1}\epsilon\theta^{l}.\label{c52}
\end{align}
Similarly, we have
\begin{align}
\|(\phi_{2}^{(l)}-\phi_{2}^{(l-1)})\|_{C^{0}}\leq M_{1}\epsilon\theta^{l}.\label{c52S}
\end{align}
$\mathbf{Step~4}:~$ The proof of~\eqref{c16} and~\eqref{c17}.\\
Now, we show the modulus of continuity for $z_{i}^{(l)}(i=1,2)$ on the temporal direction~\eqref{c16}, which is very important to prove~\eqref{c17}.\\
\indent For $\delta\in(0,1)$, we choose
\begin{align}
\Omega(\delta)=\frac{24}{1-\kappa}[A_{0}+1](\sqrt{\epsilon}\delta+\varpi(\delta|\phi'_{1_{b}})+\varpi(\delta|\phi'_{2_{b}})+\varpi(\delta|\partial_{t}\beta)).\label{c53}
\end{align}
Since $\varpi(\delta|\phi'_{i_{b}})(i=1,2)$ and $\varpi(\delta|\partial_{t}\beta)$ are monotonically increasing, bounded and continuous concave functions of $\delta$ and $\mathop{\lim}\limits_{\delta\rightarrow0^{+}}\varpi(\delta|\phi'_{i_{b}})=0$, then $\Omega(\delta)$ is also such a function and
$$\mathop{\lim}\limits_{\delta\rightarrow0^{+}}\Omega(\delta)=0.$$
At the boundary $x=L$, for any given $t_{1},t_{2}\in\mathbb{R}$ with $|t_{1}-t_{2}|\leq\delta\ll1$, one has
$$
|z_{1}^{(l)}(t_{1},L)-z_{1}^{(l)}(t_{2},L)|\leq|\phi'_{1_{br}}(t_{1})-\phi'_{1_{br}}(t_{2})|+|\kappa_{1}||z_{2_{r}}^{(l-1)}(t_{1},L)-z_{2_{r}}^{(l-1)}(t_{2},L)|.
$$
Thus, by~\eqref{c22} and~\eqref{c53}, we get
\begin{align}
\varpi(\delta|z_{1}^{(l)}(\cdot,L))&\leq \varpi(\delta|\phi'_{1_{b}})+\kappa\varpi(\delta|z_{2}^{(l-1)}(\cdot,L))\notag\\
&<\frac{1-\kappa}{24[A_{0}+1]}\Omega(\delta)+\frac{\kappa}{8[A_{0}+1]}\Omega(\delta)\notag\\
&=(\frac{1}{24[A_{0}+1]}+\frac{\kappa}{12[A_{0}+1]})\Omega(\delta)\notag\\
&<\frac{1}{8[A_{0}+1]}\Omega(\delta).\label{c54}
\end{align}
In the domain $D$, for any $x\in[0,L]$ and $t_{1},t_{2}\in\mathbb{R}$ with $|t_{1}-t_{2}|\leq\delta$, by the definition of the characteristic curve, one has
$$
t_{1}^{(l)}(y;t_{*},x)=\int_{x}^{y}\nu_{1}(\phi^{(l-1)}(t_{1}^{(l)}(\tilde{y};t_{*},x),\tilde{y})+\Phi)d\tilde{y}+t_{*}.
$$
Thus,
\begin{align*}
&|t_{1}^{(l)}(y;t_{1},x)-t_{1}^{(l)}(y;t_{2},x)|\\
\leq&|t_{1}-t_{2}|+\int_{x}^{y}|\nu_{1}(\phi^{(l-1)}(t_{1}^{(l)}(\tilde{y};t_{1},x),\tilde{y})+\Phi)-\nu_{1}(\phi^{(l-1)}(t_{1}^{(l)}(\tilde{y};t_{2},x),\tilde{y})+\Phi)|d\tilde{y}\\
\leq&|t_{1}-t_{2}|+\int_{x}^{y}\sum_{j=1}^{2}|\frac{\partial\nu_{1}}{\partial\phi_{j}}|\|z_{j}^{(l-1)}\|_{C^{0}}|t_{1}^{(l)}(\tilde{y};t_{1},x)-t_{1}^{(l)}(\tilde{y};t_{2},x)|d\tilde{y},
\end{align*}
then by the Gronwall's inequality and~\eqref{c19}, we have
\begin{align}
|t_{1}^{(l)}(y;t_{1},x)-t_{1}^{(l)}(y;t_{2},x)|\leq e^{C\epsilon}|t_{1}-t_{2}|\leq(1+\sqrt{\epsilon})|t_{1}-t_{2}|.\label{c55}
\end{align}
Noticing the concavity of $\Omega(\delta)$, we have
$$\frac{1}{1+\sqrt{\epsilon}}\Omega((1+\sqrt{\epsilon})\delta)+\frac{\sqrt{\epsilon}}{1+\sqrt{\epsilon}}\Omega(0)\leq\Omega(\delta),$$
then
$$\Omega((1+\sqrt{\epsilon})\delta)\leq(1+\sqrt{\epsilon})\Omega(\delta).$$
Thus, by~\eqref{c22}, we get
\begin{align}
&|z_{i}^{(l-1)}(t_{1}^{(l)}(y;t_{1},x),y)-z_{i}^{(l-1)}(t_{1}^{(l)}(y;t_{2},x),y)|\notag\\
\leq &\frac{1}{8[A_{0}+1]}\Omega((1+\sqrt{\epsilon})\delta)\notag\\
\leq&\frac{1}{8[A_{0}+1]}(1+\sqrt{\epsilon})\Omega(\delta).\label{c56}
\end{align}
Integrating~\eqref{c39} along $t=t_{1}^{(l)}(y;t_{1},x)$ and $t=t_{1}^{(l)}(y;t_{2},x)$ respectively and then subtracting the two results, we get
\begin{align*}
&z_{1}^{(l)}(t_{1},x)-z_{1}^{(l)}(t_{2},x)\\
=&z_{1}^{(l)}(t_{1}^{(l)}(L;t_{1},x),L)-z_{1}^{(l)}(t_{1}^{(l)}(L;t_{2},x),L)\\
&+\int_{L}^{x}\frac{\beta}{2}\nu_{1}(\Phi)z_{1}^{(l)}\Big|_{(t_{1}^{(l)}(y;t_{2},x),y)}^{(t_{1}^{(l)}(y;t_{1},x),y)}dy
+\int_{L}^{x}\frac{\beta}{2}\nu_{1}(\Phi)z_{2}^{(l-1)}\Big|_{(t_{1}^{(l)}(y;t_{2},x),y)}^{(t_{1}^{(l)}(y;t_{1},x),y)}dy\\
&+\int_{L}^{x}\frac{\partial_{t}\beta(t_{1}^{(l)}(y;t_{1},x),y)-\partial_{t}\beta(t_{1}^{(l)}(y;t_{2},x),y)}{2}\nu_{1}(\Phi)\phi_{1}^{(l)}\Big|_{(t_{1}^{(l)}(y;t_{1},x),y)}dy\\
&+\int_{L}^{x}\frac{\partial_{t}\beta(t_{1}^{(l)}(y;t_{2},x),y)}{2}\nu_{1}(\Phi)\phi_{1}^{(l)}\Big|_{(t_{1}^{(l)}(y;t_{2},x),y)}^{(t_{1}^{(l)}(y;t_{1},x),y)}dy\\
&+\int_{L}^{x}\frac{\partial_{t}\beta(t_{1}^{(l)}(y;t_{1},x),y)-\partial_{t}\beta(t_{1}^{(l)}(y;t_{2},x),y)}{2}
\Big(\nu_{1}(\phi^{(l-1)}(t_{1}^{(l)}(y;t_{1},x),y)+\Phi)\\
&-\nu_{1}(\Phi)\Big)(\phi_{1}^{(l-1)}+\phi_{2}^{(l-1)})\Big|_{(t_{1}^{(l)}(y;t_{1},x),y)}dy\\
&+\int_{L}^{x}\frac{\partial_{t}\beta(t_{1}^{(l)}(y;t_{2},x),y)}{2}\Big(\nu_{1}(\phi^{(l-1)}(t_{1}^{(l)}(y;t_{1},x),y)+\Phi)\\
&-\nu_{1}(\phi^{(l-1)}(t_{1}^{(l)}(y;t_{2},x),y)+\Phi)\Big)(\phi_{1}^{(l-1)}+\phi_{2}^{(l-1)})\Big|_{(t_{1}^{(l)}(y;t_{1},x),y)}dy\\
&+\int_{L}^{x}\frac{\partial_{t}\beta(t_{1}^{(l)}(y;t_{2},x),y)}{2}\Big(\nu_{1}(\phi^{(l-1)}(t_{1}^{(l)}(y;t_{2},x),y)+\Phi)-\nu_{1}(\Phi)\Big)\\
&(\phi_{1}^{(l-1)}+\phi_{2}^{(l-1)})\Big|_{(t_{1}^{(l)}(y;t_{2},x),y)}^{(t_{1}^{(l)}(y;t_{1},x),y)}dy\\
&+\int_{L}^{x}\frac{\partial_{t}\beta(t_{1}^{(l)}(y;t_{1},x),y)-\partial_{t}\beta(t_{1}^{(l)}(y;t_{2},x),y)}{2}\nu_{1}(\Phi)\phi_{2}^{(l-1)}\Big|_{(t_{1}^{(l)}(y;t_{1},x),y)}dy
\end{align*}
\begin{align*}
&+\int_{L}^{x}\frac{\partial_{t}\beta(t_{1}^{(l)}(y;t_{2},x),y)}{2}\nu_{1}(\Phi)\phi_{2}^{(l-1)}\Big|_{(t_{1}^{(l)}(y;t_{2},x),y)}^{(t_{1}^{(l)}(y;t_{1},x),y)}dy\\
&+\int_{L}^{x}\frac{\beta}{2}\Big(\nu_{1}(\phi^{(l-1)}(t_{1}^{(l)}(y;t_{1},x),y)+\Phi)-\nu_{1}(\phi^{(l-1)}(t_{1}^{(l)}(y;t_{2},x),y)+\Phi)\Big)\\
&(z_{1}^{(l-1)}+z_{2}^{(l-1)})\Big|_{(t_{1}^{(l)}(y;t_{1},x),y)}dy\\
&+\int_{L}^{x}\frac{\beta}{2}\Big(\nu_{1}(\phi^{(l-1)}(t_{1}^{(l)}(y;t_{2},x),y)+\Phi)-\nu_{1}(\Phi)\Big)(z_{1}^{(l-1)}+z_{2}^{(l-1)})
\Big|_{(t_{1}^{(l)}(y;t_{2},x),y)}^{(t_{1}^{(l)}(y;t_{1},x),y)}dy\\
&+\int_{L}^{x}\frac{\beta}{2}\sum_{j=1}^{2}\Big(\frac{\partial\nu_{1}}{\partial\phi_{j}}(\phi^{(l-1)}(t_{1}^{(l)}(y;t_{1},x),y)+\Phi)
-\frac{\partial\nu_{1}}{\partial\phi_{j}}(\phi^{(l-1)}(t_{1}^{(l)}(y;t_{2},x),y)+\Phi)\Big)\\
&(z_{j}^{(l-1)}\phi_{1}^{(l-1)}+z_{j}^{(l-1)}\phi_{2}^{(l-1)})\Big|_{(t_{1}^{(l)}(y;t_{1},x),y)}dy\\
&+\int_{L}^{x}\frac{\beta}{2}\sum_{j=1}^{2}\frac{\partial\nu_{1}}{\partial\phi_{j}}(\phi^{(l-1)}(t_{1}^{(l)}(y;t_{2},x),y)+\Phi)z_{j}^{(l-1)}
\Big|_{(t_{1}^{(l)}(y;t_{2},x),y)}^{(t_{1}^{(l)}(y;t_{1},x),y)}\\
&(\phi_{1}^{(l-1)}+\phi_{2}^{(l-1)})\Big|_{(t_{1}^{(l)}(y;t_{1},x),y)}dy\\
&+\int_{L}^{x}\frac{\beta}{2}\sum_{j=1}^{2}\frac{\partial\nu_{1}}{\partial\phi_{j}}(\phi^{(l-1)}(t_{1}^{(l)}(y;t_{2},x),y)+\Phi)z_{j}^{(l-1)}(t_{1}^{(l)}(y;t_{2},x),y)\\
&(\phi_{1}^{(l-1)}+\phi_{2}^{(l-1)})\Big|_{(t_{1}^{(l)}(y;t_{2},x),y)}^{(t_{1}^{(l)}(y;t_{1},x),y)}dy\\
&+\int_{L}^{x}\sum_{j=1}^{2}\Big(\frac{\partial\nu_{1}}{\partial\phi_{j}}(\phi^{(l-1)}(t_{1}^{(l)}(y;t_{1},x),y)+\Phi)
-\frac{\partial\nu_{1}}{\partial\phi_{j}}(\phi^{(l-1)}(t_{1}^{(l)}(y;t_{2},x),y)+\Phi)\Big)\\
&(z_{j}^{(l-1)}z_{1}^{(l)})\Big|_{(t_{1}^{(l)}(y;t_{1},x),y)}dy\\
&+\int_{L}^{x}\sum_{j=1}^{2}\frac{\partial\nu_{1}}{\partial\phi_{j}}(\phi^{(l-1)}(t_{1}^{(l)}(y;t_{2},x),y)+\Phi)z_{j}^{(l-1)}
\Big|_{(t_{1}^{(l)}(y;t_{2},x),y)}^{(t_{1}^{(l)}(y;t_{1},x),y)}z_{1}^{(l)}(t_{1}^{(l)}(y;t_{1},x),y)dy\\
&+\int_{L}^{x}\sum_{j=1}^{2}\frac{\partial\nu_{1}}{\partial\phi_{j}}(\phi^{(l-1)}(t_{1}^{(l)}(y;t_{2},x),y)+\Phi)z_{j}^{(l-1)}(t_{1}^{(l)}(y;t_{2},x),y)z_{1}^{(l)}
\Big|_{(t_{1}^{(l)}(y;t_{2},x),y)}^{(t_{1}^{(l)}(y;t_{1},x),y)}dy.
\end{align*}
By~\eqref{a2}-\eqref{a3}, \eqref{b13}, \eqref{c19}, \eqref{c54}-\eqref{c56} and Gronwall's inequality, we have
\begin{align}
|z_{1}^{(l)}(t_{1},x)-z_{1}^{(l)}(t_{2},x)|\leq\frac{1}{8[A_{0}+1]}C_{2}\Omega(\delta),\label{c57}
\end{align}
where $C_{2}>0$ is a constant independent of $l$.\\
\indent Using the integral expression~\eqref{c46} of $F_{1}(t,x)z_{1}^{(l)}(t,x)$ and then by~\eqref{a2}-\eqref{a3}, \eqref{b13}, \eqref{c19}, and~\eqref{c54}-\eqref{c57},we have
\begin{align}
&|z_{1}^{(l)}(t_{1},x)-z_{1}^{(l)}(t_{2},x)|\notag\\
\leq&\frac{1}{F_{1}(t,x)}(\frac{1}{24[A_{0}+1]}+\frac{1}{12[A_{0}+1]}\kappa)(1+\sqrt{\epsilon})\Omega(\delta)\notag\\
&+\frac{F_{1}(t,x)-1}{F_{1}(t,x)}\frac{1}{8[A_{0}+1]}(1+\sqrt{\epsilon})\Omega(\delta)
+C\epsilon^{2}(1+\sqrt{\epsilon})|t_{1}-t_{2}|\notag\\
&+C\epsilon\frac{1+\sqrt{\epsilon}}{8[A_{0}+1]}\Omega(\delta)+\frac{1}{8[A_{0}+1]}C\epsilon\Omega(\delta)\notag\\
\leq&\frac{1}{8[A_{0}+1]}\Omega(\delta)-(\frac{1-\kappa}{12[A_{0}+1]}-\sqrt{\epsilon})\Omega(\delta)\notag\\
<&\frac{1}{8[A_{0}+1]}\Omega(\delta).\label{c58}
\end{align}
Finally, we prove~\eqref{c17}. We first consider the special case that two given points $(t_{1},x_{1})$ and $(t_{2},x_{2})$ with $|t_{1}-t_{2}|\leq\delta,|x_{1}-x_{2}|\leq\delta$ locate on the same characteristic curve $t=t_{1}^{(l)}(x;t_{0},x_{0})$, namely, $t_{2}=t_{1}^{(l)}(x_{2};t_{1},x_{1})$. Using the similar method of~\eqref{c47}, we can get
\begin{align}
|z_{1}^{(l)}(t_{1},x_{1})-z_{1}^{(l)}(t_{2},x_{2})|\leq C\epsilon\delta\leq\frac{1}{12}\Omega(\delta).\label{c59}
\end{align}
Then, for general two points $(t_{1},x_{1})$ and $(t_{2},x_{2})$ with $|t_{1}-t_{2}|\leq\delta,|x_{1}-x_{2}|\leq\delta$, we can choose a point $(t_{3},x_{1})$ locating on the $1$st characteristic curve passing through $(t_{2},x_{2})$, namely, $t_{3}=t_{1}^{(l)}(x_{1};t_{2},x_{2})$.\\
\indent By definition~\eqref{c24} and~\eqref{b13}, we have
$$|t_{3}-t_{2}|\leq|\nu_{1}||x_{1}-x_{2}|\leq A_{0}\delta,$$
and thus
$$|t_{3}-t_{1}|\leq|t_{3}-t_{2}|+|t_{2}-t_{1}|\leq(A_{0}+1)\delta.$$
Now we combine estimates~\eqref{c58}-\eqref{c59} to get
\begin{align}
&|z_{1}^{(l)}(t_{1},x_{1})-z_{1}^{(l)}(t_{2},x_{2})|\notag\\
\leq&|z_{1}^{(l)}(t_{1},x_{1})-z_{1}^{(l)}(\frac{[A_{0}+1]t_{1}+t_{3}}{A_{0}+1},x_{1})|\notag\\
&+|z_{1}^{(l)}(\frac{[A_{0}+1]t_{1}+t_{3}}{A_{0}+1},x_{1})
-z_{1}^{(l)}(\frac{([A_{0}+1]-1)t_{1}+2t_{3}}{A_{0}+1},x_{1})|\notag\\
&+\ldots+|z_{1}^{(l)}(\frac{t_{1}+[A_{0}+1]t_{3}}{A_{0}+1},x_{1})-z_{1}^{(l)}(t_{3},x_{1})|+|z_{1}^{(l)}(t_{3},x_{1})-z_{1}^{(l)}(t_{2},x_{2})|\notag\\
\leq&\frac{[A_{0}+1]+1}{8[A_{0}+1]}\Omega(\delta)+\frac{1}{12}\Omega(\delta)\notag\\
\leq&\frac{1}{3}\Omega(\delta).\label{c60}
\end{align}
The combination of~\eqref{c59} and~\eqref{c60} leads to
\begin{align}
\varpi(\delta|z_{1}^{(l)})\leq\frac{1}{3}\Omega(\delta).\label{c61}
\end{align}
In a similar way, we obtain
\begin{align}
\varpi(\delta|z_{2}^{(l)})\leq\frac{1}{3}\Omega(\delta).\label{c62}
\end{align}
By the aid of equations~\eqref{c3}, \eqref{c5} and by~\eqref{a2}, \eqref{b13}, \eqref{c19}, \eqref{c23}, \eqref{c53} and~\eqref{c61}-\eqref{c62}, we have
\begin{align}
\varpi(\delta|\partial_{x}\phi_{i}^{(l)})&\leq A_{0}\frac{1}{3}\Omega(\delta)+C\epsilon^{2}\delta+C\epsilon\delta\notag\\
&\leq A_{0}\frac{1}{2}\Omega(\delta),\quad i=1,2.\label{c63}
\end{align}
Thus, \eqref{c61}-\eqref{c63} indicate~\eqref{c17}.
\end{proof}
\indent With the help of~\proref{p1} and the similar arguments as in \cite{Qu}, the proof of Theorem 2.1 could be presented, here we omit the details.
\section{Stability of the Time-periodic Solution}\label{s4}
\indent\indent In this section, we give the proof of~\theref{t2} to consider the stability of the time-periodic solution obtained in Theorem 2.1. For the sake of proving the existence of the classical solutions $\phi=\phi(t,x)$, we only need to prove the following ~\lemref{l1} on the basis of the existence and uniqueness of local $C^{1}$ solution for the mixed initial-boundary value problem for quasilinear hyperbolic system(cf. Chapter $4$ in \cite{Yuw}).
Using the method in~\cite{Li}, we can give the proof of~\lemref{l1}. Here we omit the details.
\begin{lemma}\label{l1}
 There exists a small constant $\varepsilon_{6}>0$, for any given $\varepsilon\in(0,\varepsilon_{6})$, there exists $\sigma=\sigma(\varepsilon)>0$ such that if
 \begin{align*}
 &\|\phi_{i_{b}}\|_{C^{1}(\mathbb{R}_{+})}\leq\sigma,\quad i=1,2,\\
 &\|\phi_{0}\|_{C^{1}[0,L]}\leq\sigma,
 \end{align*}
 then the $C^{1}$ solution $\phi=\phi(t,x)$ to the initial-boundary value problem~\eqref{b7}-\eqref{b10} satisfies
\begin{align}
\|\phi\|_{C^{1}(D)}\leq \varepsilon.\label{d1}
\end{align}
\end{lemma}
Now, we prove~\eqref{b17} inductively. For any $t_{*}>0$ and $N\in\mathbb{N}$, we prove
\begin{align}
\mathop{\max}\limits_{i=1,2}\|\phi_{i}(t,\cdot)-\phi_{i}^{(T_{*})}(t,\cdot)\|_{C^{0}}\leq C_{S}\epsilon\xi^{N+1},\quad\forall t\in[t_{*}+T_{0},t_{*}+2T_{0}],\label{d2}
\end{align}
under the hypothesis
\begin{align}
\mathop{\max}\limits_{i=1,2}\|\phi_{i}(t,\cdot)-\phi_{i}^{(T_{*})}(t,\cdot)\|_{C^{0}}\leq C_{S}\epsilon\xi^{N},\quad\forall t\in[t_{*},t_{*}+T_{0}], \label{d3}
\end{align}
where $\xi\in(0,1)$ is a constant to be determined later and $\phi_{i}^{(T_{*})}(t,x), i=1,2$ is the time-periodic solution obtained in~\theref{t1}. Let
$$
\theta(t)=\mathop{\max}\limits_{1\leq i\leq2}\mathop{\sup}\limits_{x\in[0,L]}|\phi_{i}(t,x)-\phi_{i}^{(T_{*})}(t,x)|.
$$
Obviously, $\theta(t)$ is continuous and
$$
\theta(t_{*}+T_{0})\leq C_{S}\varepsilon\xi^{N}
$$
follows~\eqref{d3}.
Then it's just necessary to prove
\begin{align}
\theta(t)\leq C_{S}\epsilon\xi^{N+1},\quad\forall t\in[t_{*}+T_{0},\tau]\label{d4}
\end{align}
under the assumption
\begin{align}
\theta(t)\leq C_{S}\epsilon\xi^{N},\quad\forall t\in[t_{*},\tau]\label{d5}
\end{align}
for any $\tau\in[t_{*}+T_{0},t_{*}+2T_{0}]$.\\
\indent At the boundary $x=L$, one has
$$
\phi_{1}(t,L)-\phi_{1}^{(T_{*})}(t,L)=\kappa_{1}(\phi_{2}(t,L)-\phi_{2}^{(T_{*})}(t,L)),
$$
then from~\eqref{d5}, we have
\begin{align}
|\phi_{1}(t,L)-\phi_{1}^{(T_{*})}(t,L)|\leq|\kappa_{1}|C_{S}\epsilon\xi^{N}.\label{d6}
\end{align}
Similarly, at the boundary $x=0$, we get
\begin{align}
|\phi_{2}(t,0)-\phi_{2}^{(T_{*})}(t,0)|\leq|\kappa_{2}|C_{S}\epsilon\xi^{N}.\label{d7}
\end{align}
\indent As for the interior estimates, we have
\begin{align}
\partial_{x}\phi_{1}^{(T_{*})}+\nu_{1}(\phi^{(T_{*})}+\Phi)\partial_{t}\phi_{1}^{(T_{*})}=&\frac{\beta(t,x)}{2}\nu_{1}(\Phi)\phi_{1}^{(T_{*})}
+\frac{\beta(t,x)}{2}(\nu_{1}(\phi^{(T_{*})}+\Phi)\notag\\
&-\nu_{1}(\Phi))\phi_{1}^{(T_{*})}+\frac{\beta(t,x)}{2}\nu_{1}(\phi^{(T_{*})}+\Phi)\phi_{2}^{(T_{*})}.\label{d8}
\end{align}
Then by~\eqref{c1} and~\eqref{d8}, we get
\begin{align}
&\partial_{x}(\phi_{1}-\phi_{1}^{(T_{*})})+\nu_{1}(\phi+\Phi)\partial_{t}(\phi_{1}-\phi_{1}^{(T_{*})})\notag\\
=&-\Big(\nu_{1}(\phi+\Phi)-\nu_{1}(\phi^{(T_{*})}+\Phi)\Big)\partial_{t}\phi_{1}^{(T_{*})}+\frac{\beta(t,x)}{2}\nu_{1}(\Phi)(\phi_{1}-\phi_{1}^{(T_{*})})\notag\\
&+\frac{\beta(t,x)}{2}\Big(\nu_{1}(\phi+\Phi)-\nu_{1}(\Phi)\Big)\Big((\phi_{1}+\phi_{2})-(\phi_{1}^{(T_{*})}+\phi_{2}^{(T_{*})})\Big)\notag\\
&+\frac{\beta(t,x)}{2}\Big(\nu_{1}(\phi+\Phi)-\nu_{1}(\phi^{(T_{*})}+\Phi)\Big)(\phi_{1}^{(T_{*})}+\phi_{2}^{(T_{*})})\notag\\
&+\frac{\beta(t,x)}{2}\nu_{1}(\Phi)(\phi_{2}-\phi_{2}^{(T_{*})}).\label{d10}
\end{align}
\indent Multiplying $F_{1}(t,x)$ on both sides of~\eqref{d10}, we gain
\begin{align*}
&\partial_{x}\Big(F_{1}(\phi_{1}-\phi_{1}^{(T_{*})})\Big)+\nu_{1}(\phi+\Phi)\partial_{t}\Big(F_{1}(\phi_{1}-\phi_{1}^{(T_{*})})\Big)\\
=&-F_{1}\Big(\nu_{1}(\phi+\Phi)-\nu_{1}(\phi^{(T_{*})}+\Phi)\Big)\partial_{t}\phi_{1}^{(T_{*})}\\
&+\nu_{1}(\phi+\Phi)F_{1}(\phi_{1}-\phi_{1}^{(T_{*})})\int_{x}^{L}\frac{\partial_{t}\beta(t,s)}{2}\nu_{1}(\Phi)ds\\
&+\frac{\beta(t,x)}{2}F_{1}\Big(\nu_{1}(\phi+\Phi)-\nu_{1}(\Phi)\Big)\Big((\phi_{1}+\phi_{2})-(\phi_{1}^{(T_{*})}+\phi_{2}^{(T_{*})})\Big)\\
&+\frac{\beta(t,x)}{2}F_{1}\Big(\nu_{1}(\phi+\Phi)-\nu_{1}(\phi^{(T_{*})}+\Phi)\Big)(\phi_{1}^{(T_{*})}+\phi_{2}^{(T_{*})})\\
&+\frac{\beta(t,x)}{2}F_{1}\nu_{1}(\Phi)(\phi_{2}-\phi_{2}^{(T_{*})}),
\end{align*}
then integrate the result along the $1$st characteristic curve $t=t_{1}(x;\hat{t},\hat{x})$ defined by
\begin{align}
\left\{
\begin{aligned}
&\frac{dt_{1}}{dx}(x;\hat{t},\hat{x})=\nu_{1}(\phi(t_{1}(x;\hat{t},\hat{x}),x)+\Phi),\\
&t_{1}(\hat{x};\hat{t},\hat{x})=\hat{t}.
\end{aligned}\right.\label{d11}
\end{align}
Noting $T_{0}=L\mathop{\max}\limits_{i=1,2}\mathop{\sup}\limits_{\phi\in \Psi}|\nu_{i}(\phi+\Phi)|$, for each points $(\hat{t},\hat{x})\in[t_{*}+T_{0},\tau]\times[0,L]$, the backward curve $t=t_{1}(x;\hat{t},\hat{x})$ will intersect the boundary in a time interval shorter than $T_{0}$, namely,
$$t_{1}(L;\hat{t},\hat{x})\in[\hat{t}-T_{0},\hat{t}]\subseteq[t_{*},\tau],\quad \forall(\hat{t},\hat{x})\in[t_{*}+T_{0},\tau]\times[0,L],$$
and thus we can use estimates~\eqref{d6}-\eqref{d7} on the boundary.\\
\indent Using~\eqref{a2}, \eqref{b16} and~\eqref{d5}-\eqref{d6}, we get
\begin{align*}
&|\phi_{1}(\hat{t},\hat{x})-\phi_{1}^{(T_{*})}(\hat{t},\hat{x})|\\
\leq&\frac{|\kappa_{1}|}{F_{1}(\hat{t},\hat{x})}C_{S}\epsilon\xi^{N}+C\epsilon C_{S}\epsilon\xi^{N}+\frac{F_{1}(\hat{t},\hat{x})-1}{F_{1}(\hat{t},\hat{x})}C_{S}\epsilon\xi^{N}\\
\leq& C_{S}\epsilon\xi^{N+1}.
\end{align*}
Here we choose a constant $0<\xi<1$ satisfying
$$ \frac{|\kappa_{1}|}{F_{1}(\hat{t},\hat{x})}+C\epsilon+ \frac{F_{1}(\hat{t},\hat{x})-1}{F_{1}(\hat{t},\hat{x})}\leq\xi.$$
Similarly, we get
\begin{align*}
&|\phi_{2}(\hat{t},\hat{x})-\phi_{2}^{(T_{*})}(\hat{t},\hat{x})|\leq C_{S}\epsilon\xi^{N+1}.
\end{align*}
Thus, we have
$$\theta(\hat{t})\leq C_{S}\epsilon\xi^{N+1}.$$
Since $\hat{t}\in[t_{*}+T_{0},\tau]$ is arbitrary, we get~\eqref{d4}.
\section{Regularity of the Time-periodic Solution}\label{s5}
\indent\indent In this section, we will prove higher regularity of the time-periodic solutions, provided that all boundary functions $m_{i_{b}}(t)(i=1,2)$ possess higher regularity.\\
\indent In order to get the regularity of $\phi^{(T_{*})}$, we use the iteration scheme~\eqref{c3}-\eqref{c6} introduced in~\secref{s3} and prove the following proposition.
\begin{proposition}\label{p2}
For the iteration scheme~\eqref{c3}-\eqref{c6}, assuming that~\eqref{b18} holds, then exists a large enough constant~$C_{R}>0$, such that for any given $l\in \mathbb{N}_{+}$, we have
\begin{align}
&\|\partial_{t}^{2}\phi_{i}^{(l)}\|_{L^{\infty}}\leq C_{R},\label{e1}\\
&\|\partial_{t}\partial_{x}\phi_{i}^{(l)}\|_{L^{\infty}}\leq A_{0}C_{R},\label{e2}\\
&\|\partial_{x}^{2}\phi_{i}^{(l)}\|_{L^{\infty}}\leq A_{0}^{2}C_{R}\label{e3}
\end{align}
under the hypothesis
\begin{align}
&\|\partial_{t}^{2}\phi_{i}^{(l-1)}\|_{L^{\infty}}\leq C_{R}<+\infty,\label{e4}\\
&\|\partial_{t}\partial_{x}\phi_{i}^{(l-1)}\|_{L^{\infty}}\leq A_{0}C_{R},\label{e5}\\
&\|\partial_{x}^{2}\phi_{i}^{(l-1)}\|_{L^{\infty}}\leq A_{0}^{2}C_{R}.\label{e6}
\end{align}
\end{proposition}
\begin{proof}
\indent Since we use actually the same sequence constructed in~\secref{s3}. By~\proref{p1}, we already have~\eqref{c13}-\eqref{c17} for each $l$ and especially,
\begin{align}
\|\phi^{(l)}\|_{C^{1}}\leq(C_{1}+M_{1})\epsilon,\quad\|\phi^{(l-1)}\|_{C^{1}}\leq(C_{1}+M_{1})\epsilon.\label{e7}
\end{align}
Let
$$\varphi_{i}^{(l)}=\partial_{t}z_{i}^{(l)}=\partial_{t}^{2}\phi_{i}^{(l)},\quad i=1,2; l\in \mathbb{N}_{+}.$$
First, by the aid of boundary conditions~\eqref{c40} and~\eqref{c42} and using~\eqref{b18} and~\eqref{e4}, we have
\begin{align}
&|\varphi_{1}^{(l)}(t,L)|\leq M_{2}+|\kappa_{1}|C_{R},\label{e8}\\
&|\varphi_{2}^{(l)}(t,0)|\leq M_{2}+|\kappa_{2}|C_{R}.\label{e9}
\end{align}
Then we use a method similar to~\eqref{c47} to get
\begin{align}
\|\varphi_{1}^{(l)}(t,x)\|_{L^{\infty}}\leq \hbar_{1}C_{R},\label{e10}
\end{align}
where constant $0<\hbar_{1}<1$ is independent of $l$.\\
\indent Similarly, we have
\begin{align}
\|\varphi_{2}^{(l)}(t,x)\|_{L^{\infty}}\leq \hbar_{1}C_{R}.\label{e11}
\end{align}
Next, using~\eqref{c39} and~\eqref{c41} and by~\eqref{a2}-\eqref{a3}, \eqref{b13}, \eqref{c13} and~\eqref{e10}-\eqref{e11}, we get
\begin{align}
\|\partial_{x}\partial_{t}\phi_{i}^{(l)}|_{L^{\infty}}&\leq A_{0}\hbar_{1}C_{R}+C\epsilon+C\epsilon^{2}\notag\\
&\leq A_{0}\hbar_{2}C_{R},\label{e12}
\end{align}
where $\hbar_{2}\in(\hbar_{1},1)$ is a constant independent of $l$.\\
\indent Then taking the spatial derivative to~\eqref{c3}, we have
\begin{align*}
\partial_{x}^{2}\phi_{1}^{(l)}=&-\nu_{1}(\phi^{(l-1)}+\Phi)\partial_{x}\partial_{t}\phi_{1}^{(l)}-\sum_{j=1}^{2}\frac{\partial\nu_{1}}{\partial\phi_{j}}(\phi^{(l-1)}+\Phi)
\partial_{x}\phi_{j}^{(l-1)}z_{1}^{(l)}\\
&+\frac{\beta}{2}\nu_{1}(\Phi)\partial_{x}\phi_{1}^{(l)}+\frac{\beta}{2}\sum_{j=1}^{2}\frac{\partial\nu_{1}}{\partial\phi_{j}}
(\phi^{(l-1)}+\Phi)\partial_{x}\phi_{j}^{(l-1)}\phi_{1}^{(l-1)}\\
&+\frac{\partial_{x}\beta}{2}\nu_{1}(\Phi)\phi_{1}^{(l)}
+\frac{\beta}{2}(\nu_{1}(\phi^{(l-1)}+\Phi)-\nu_{1}(\Phi))\partial_{x}\phi_{1}^{(l-1)}\\
&+\frac{\partial_{x}\beta}{2}\Big(\nu_{1}(\phi^{(l-1)}+\Phi)-\nu_{1}(\Phi)\Big)\phi_{1}^{(l-1)}+\frac{\beta}{2}\nu_{1}(\phi^{(l-1)}+\Phi)\partial_{x}\phi_{2}^{(l-1)}\\
&+\frac{\beta}{2}\sum_{j=1}^{2}\frac{\partial\nu_{1}}{\partial\phi_{j}}(\phi^{(l-1)}+\Phi)
\partial_{x}\phi_{j}^{(l-1)}\phi_{2}^{(l-1)}+\frac{\partial_{x}\beta}{2}\nu_{1}(\phi^{(l-1)}+\Phi)\phi_{2}^{(l-1)}.
\end{align*}
Thus, using~\eqref{a2}-\eqref{a3}, \eqref{b13}, \eqref{c13}-\eqref{c14} \eqref{e7} and~\eqref{e12}, we get
\begin{align}
\|\partial_{x}^{2}\phi_{1}^{(l)}\|_{L^{\infty}}\leq& A_{0}^{2}\hbar_{2}C_{R}+C\epsilon\notag\\
\leq& A_{0}^{2}C_{R}.\label{e13}
\end{align}
In a similar way, we also get
\begin{align}
\|\partial_{x}^{2}\phi_{2}^{(l)}\|_{L^{\infty}}\leq A_{0}^{2}C_{R}.\label{e14}
\end{align}
\qedhere
\end{proof}
$\mathbf{The~Proof~of~Theorem~2.3.}$ By~\eqref{e1}-\eqref{e3}, we know that $\{\phi^{(l)}\}_{l=1}^{\infty}$ is uniformly $W^{2,\infty}$ bounded and then $weak^{*}$ convergent. Moreover, noting that $\{\phi^{(l)}\}_{l=1}^{\infty}$ converges strongly to $\phi^{(T_{*})}$ in $C^{1}$, so we get the $W^{2,\infty}$ regularity of $\phi^{(T_{*})}$.
\section{Boundary Stabilization around the Time-periodic Solution}\label{s6}
\indent\indent In this section, we will give the proof of~\theref{t5}.\\
\indent Noting that we have got the $C^{0}$ exponential convergence in~\theref{t2} as follows:
\begin{align}
\|\phi_{i}(t,\cdot)-\phi_{i}^{(T_{*})}(t,\cdot)\|_{C^{0}}\leq C_{S}\epsilon\xi^{N},\quad\forall t\in[NT_{0},(N+1)T_{0}),\forall l\in \mathbb{N}_{+},\label{f1}
\end{align}
which also shows that
\begin{align}
\|\phi_{i}(t,\cdot)-\phi_{i}^{(T_{*})}(t,\cdot)\|_{C^{0}}\leq C_{S}\epsilon\xi^{N+1},\quad\forall t\in[(N+1)T_{0},(N+2)T_{0}),\forall l\in \mathbb{N}_{+}.\label{f2}
\end{align}
Moreover, by~\theref{t1}, \theref{t4} and~\lemref{l1}, we have
\begin{align}
\|\phi\|_{C^{1}}\leq C_{3}\epsilon,~~\|\phi^{(T_{*})}\|_{C^{1}}\leq C_{E}\epsilon,~~\|\phi^{(T_{*})}\|_{W^{2,\infty}}\leq(1+A_{0})^{2}C_{R}.\label{f3}
\end{align}
By the continuity, we will inductively get the estimates for the convergence of the first derivatives, namely, for each $N\in \mathbb{N}_{+}$ and $\tau\in[(N+1)T_{0},(N+2)T_{0}]$, we will prove
\begin{align}
&\|\partial_{t}\phi_{i}(t,\cdot)-\partial_{t}\phi_{i}^{(T_{*})}(t,\cdot)\|_{C^{0}}\leq C_{S}^{*}\epsilon\xi^{N+1},\quad\forall t\in[(N+1)T_{0},\tau],\forall l\in \mathbb{N}_{+},\label{f4}\\
&\|\partial_{x}\phi_{i}(t,\cdot)-\partial_{x}\phi_{i}^{(T_{*})}(t,\cdot)\|_{C^{0}}\leq A_{0}C_{S}^{*}\epsilon\xi^{N+1},\quad\forall t\in[(N+1)T_{0},\tau],\forall l\in \mathbb{N}_{+}\label{f5}
\end{align}
under the assumption
\begin{align}
&\|\partial_{t}\phi_{i}(t,\cdot)-\partial_{t}\phi_{i}^{(T_{*})}(t,\cdot)\|_{C^{0}}\leq C_{S}^{*}\epsilon\xi^{N},\quad\forall t\in[NT_{0},\tau],\forall l\in \mathbb{N}_{+},\label{f6}\\
&\|\partial_{x}\phi_{i}(t,\cdot)-\partial_{x}\phi_{i}^{(T_{*})}(t,\cdot)\|_{C^{0}}\leq A_{0}C_{S}^{*}\epsilon\xi^{N},\quad\forall t\in[NT_{0},\tau],\forall l\in \mathbb{N}_{+}.\label{f7}
\end{align}
Let
$$z_{i}=\partial_{t}\phi_{i},\quad w_{i}=\partial_{x}\phi_{i}$$
and
$$z_{i}^{(T_{*})}=\partial_{t}\phi_{i}^{(T_{*})},\quad w_{i}^{(T_{*})}=\partial_{x}\phi_{i}^{(T_{*})}.$$
Taking the temporal derivative on boundary conditions~\eqref{b9}-\eqref{b10}, we get
\begin{align*}
&z_{2}(t,0)=\phi'_{2_{b}}(t)+\kappa_{2}z_{1}(t,0),\quad t>0,\\
&z_{1}(t,L)=\phi'_{1_{b}}(t)+\kappa_{1}z_{2}(t,L),\quad t>0
\end{align*}
and
\begin{align*}
&z_{2}^{(T_{*})}(t,0)=\phi'_{2_{b}}(t)+\kappa_{2}z_{1}^{(T_{*})}(t,0),\quad t>0,\\
&z_{1}^{(T_{*})}(t,L)=\phi'_{1_{b}}(t)+\kappa_{1}z_{2}^{(T_{*})}(t,L),\quad t>0.
\end{align*}
Thus, on the boundary $x=0$, we have
\begin{align}
\mathop{\sup}\limits_{t\in[NT_{0},\tau]}|z_{2}(t,0)-z_{2}^{(T_{*})}(t,0)|&\leq|\kappa_{2}||z_{1}(t,0)-z_{1}^{(T_{*})}(t,0)|\notag\\
&\leq|\kappa_{2}|C_{S}^{*}\epsilon\xi^{N}.\label{f8}
\end{align}
Similarly, on $x=L$, we have
\begin{align}
\mathop{\sup}\limits_{t\in[NT_{0},\tau]}|z_{1}(t,L)-z_{1}^{(T_{*})}(t,L)|\leq|\kappa_{1}|C_{S}^{*}\epsilon\xi^{N}.\label{f9}
\end{align}
In the domain, we take the temporal derivative of~\eqref{c1} and~\eqref{d8} to get
\begin{align}
&\partial_{x}z_{1}+\nu_{1}(\phi+\Phi)\partial_{t}z_{1}\notag\\
=&\frac{\beta}{2}\nu_{1}(\Phi)z_{1}+\frac{\partial_{t}\beta}{2}\nu_{1}(\Phi)\phi_{1}+\frac{\beta}{2}\nu_{1}(\Phi)z_{2}+\frac{\partial_{t}\beta}{2}\nu_{1}(\Phi)\phi_{2}\notag\\
&+\frac{\beta}{2}\Big(\nu_{1}(\phi+\Phi)-\nu_{1}(\Phi)\Big)(z_{1}+z_{2})\notag\\
&+\frac{\partial_{t}\beta}{2}\Big(\nu_{1}(\phi+\Phi)-\nu_{1}(\Phi)\Big)(\phi_{1}+\phi_{2})\notag\\
&+\frac{\beta}{2}\sum_{j=1}^{2}\frac{\partial\nu_{1}}{\partial\phi_{j}}(z_{j}\phi_{1}+z_{j}\phi_{2})
-\sum_{j=1}^{2}\frac{\partial\nu_{1}}{\partial\phi_{j}}z_{j}z_{1},\label{f10}
\end{align}
\begin{align}
&\partial_{x}z_{1}^{(T_{*})}+\nu_{1}(\phi^{(T_{*})}+\Phi)\partial_{t}z_{1}^{(T_{*})}\notag\\
=&\frac{\beta}{2}\nu_{1}(\Phi)z_{1}^{(T_{*})}+\frac{\partial_{t}\beta}{2}\nu_{1}(\Phi)\phi_{1}^{(T_{*})}+\frac{\beta}{2}\nu_{1}(\Phi)z_{2}^{(T_{*})}
+\frac{\partial_{t}\beta}{2}\nu_{1}(\Phi)\phi_{2}^{(T_{*})}\notag\\
&+\frac{\beta}{2}\Big(\nu_{1}(\phi^{(T_{*})}+\Phi)-\nu_{1}(\Phi)\Big)(z_{1}^{(T_{*})}+z_{2}^{(T_{*})})\notag\\
&+\frac{\partial_{t}\beta}{2}\Big(\nu_{1}(\phi^{(T_{*})}+\Phi)-\nu_{1}(\Phi)\Big)(\phi_{1}^{(T_{*})}+\phi_{2}^{(T_{*})})\notag\\
&+\frac{\beta}{2}\sum_{j=1}^{2}\frac{\partial\nu_{1}}{\partial\phi_{j}}(z_{j}^{(T_{*})}\phi_{1}^{(T_{*})}+z_{j}^{(T_{*})}\phi_{2}^{(T_{*})})
-\sum_{j=1}^{2}\frac{\partial\nu_{1}}{\partial\phi_{j}}z_{j}^{(T_{*})}z_{1}^{(T_{*})}.\label{f11}
\end{align}
Furthermore, by~\eqref{f10}-\eqref{f11}, we get
\begin{align}
&\partial_{x}(z_{1}-z_{1}^{(T_{*})})+\nu_{1}(\phi+\Phi)\partial_{t}(z_{1}-z_{1}^{(T_{*})})\notag\\
=&-\Big(\nu_{1}(\phi+\Phi)-\nu_{1}(\phi^{(T_{*})}+\Phi)\Big)\partial_{t}z_{1}^{(T_{*})}+\frac{\beta}{2}\nu_{1}(\Phi)(z_{1}-z_{1}^{(T_{*})})\notag\\
&+\frac{\partial_{t}\beta}{2}\nu_{1}(\Phi)(\phi_{1}-\phi_{1}^{(T_{*})})+\frac{\beta}{2}\nu_{1}(\Phi)(z_{2}-z_{2}^{(T_{*})})
+\frac{\partial_{t}\beta}{2}\nu_{1}(\Phi)(\phi_{2}-\phi_{2}^{(T_{*})})\notag\\
&+\frac{\beta}{2}\Big(\nu_{1}(\phi+\Phi)-\nu_{1}(\Phi)\Big)\Big((z_{1}+z_{2})-(z_{1}^{(T_{*})}+z_{2}^{(T_{*})})\Big)\notag\\
&+\frac{\beta}{2}\Big(\nu_{1}(\phi+\Phi)-\nu_{1}(\phi^{(T_{*})}+\Phi)\Big)(z_{1}^{(T_{*})}+z_{2}^{(T_{*})})\notag\\
&+\frac{\partial_{t}\beta}{2}\Big(\nu_{1}(\phi+\Phi)-\nu_{1}(\Phi)\Big)\Big((\phi_{1}+\phi_{2})-(\phi_{1}^{(T_{*})}+\phi_{2}^{(T_{*})})\Big)\notag\\
&+\frac{\partial_{t}\beta}{2}\Big(\nu_{1}(\phi+\Phi)-\nu_{1}(\phi^{(T_{*})}+\Phi)\Big)(\phi_{1}^{(T_{*})}+\phi_{2}^{(T_{*})})\notag\\
&+\frac{\beta}{2}\sum_{j=1}^{2}\frac{\partial\nu_{1}}{\partial\phi_{j}}(\phi+\Phi)z_{j}\Big((\phi_{1}+\phi_{2})-(\phi_{1}^{(T_{*})}+\phi_{2}^{(T_{*})})\Big)\notag\\
&+\frac{\beta}{2}\sum_{j=1}^{2}\frac{\partial\nu_{1}}{\partial\phi_{j}}(\phi+\Phi)(z_{j}-z_{j}^{(T_{*})})(\phi_{1}^{(T_{*})}+\phi_{2}^{(T_{*})})\notag\\
&+\frac{\beta}{2}\sum_{j=1}^{2}\Big(\frac{\partial\nu_{1}}{\partial\phi_{j}}(\phi+\Phi)-\frac{\partial\nu_{1}}{\partial\phi_{j}}(\phi^{(T_{*})}+\Phi)\Big)
(z_{j}^{(T_{*})})(\phi_{1}^{(T_{*})}+\phi_{2}^{(T_{*})})\notag\\
&-\sum_{j=1}^{2}\frac{\partial\nu_{1}}{\partial\phi_{j}}(\phi+\Phi)z_{j}(z_{1}-z_{1}^{(T_{*})})
-\sum_{j=1}^{2}\frac{\partial\nu_{1}}{\partial\phi_{j}}(\phi+\Phi)(z_{j}-z_{j}^{(T_{*})})z_{1}^{(T_{*})}\notag\\
&-\sum_{j=1}^{2}\Big(\frac{\partial\nu_{1}}{\partial\phi_{j}}(\phi+\Phi)-\frac{\partial\nu_{1}}{\partial\phi_{j}}(\phi^{(T_{*})}+\Phi)\Big)z_{j}^{(T_{*})}z_{1}^{(T_{*})}.\label{f12}
\end{align}
And we multiply $F_{1}(t,x)$ on both sides of~\eqref{f12} to gain
\begin{align*}
&\partial_{x}\Big(F_{1}(z_{1}-z_{1}^{(T_{*})})\Big)+\nu_{1}(\phi+\Phi)\partial_{t}\Big(F_{1}(z_{1}-z_{1}^{(T_{*})})\Big)\notag\\
=&-F_{1}\Big(\nu_{1}(\phi+\Phi)-\nu_{1}(\phi^{(T_{*})}+\Phi)\Big)\partial_{t}z_{1}^{(T_{*})}\notag\\
&+\nu_{1}(\phi+\Phi)F_{1}(z_{1}-z_{1}^{(T_{*})})\int_{x}^{L}\frac{\partial_{t}\beta(t,s)}{2}\nu_{1}(\Phi)ds\notag\\
&+\frac{\partial_{t}\beta}{2}F_{1}\nu_{1}(\Phi)(\phi_{1}-\phi_{1}^{(T_{*})})+\frac{\partial F_{1}}{\partial x}(z_{2}-z_{2}^{(T_{*})})\notag\\
&+\frac{\partial_{t}\beta}{2}F_{1}\nu_{1}(\Phi)(\phi_{2}-\phi_{2}^{(T_{*})})\notag\\
&+\frac{\beta}{2}F_{1}\Big(\nu_{1}(\phi+\Phi)-\nu_{1}(\Phi)\Big)\Big((z_{1}+z_{2})-(z_{1}^{(T_{*})}+z_{2}^{(T_{*})})\Big)\notag\\
&+\frac{\beta}{2}F_{1}\Big(\nu_{1}(\phi+\Phi)-\nu_{1}(\phi^{(T_{*})}+\Phi)\Big)(z_{1}^{(T_{*})}+z_{2}^{(T_{*})})\notag\\
&+\frac{\partial_{t}\beta}{2}F_{1}\Big(\nu_{1}(\phi+\Phi)-\nu_{1}(\Phi)\Big)\Big((\phi_{1}+\phi_{2})-(\phi_{1}^{(T_{*})}+\phi_{2}^{(T_{*})})\Big)\notag\\
&+\frac{\partial_{t}\beta}{2}F_{1}\Big(\nu_{1}(\phi+\Phi)-\nu_{1}(\phi^{(T_{*})}+\Phi)\Big)(\phi_{1}^{(T_{*})}+\phi_{2}^{(T_{*})})\notag\\
&+\frac{\beta}{2}F_{1}\sum_{j=1}^{2}\frac{\partial\nu_{1}}{\partial\phi_{j}}(\phi+\Phi)z_{j}\Big((\phi_{1}+\phi_{2})-(\phi_{1}^{(T_{*})}+\phi_{2}^{(T_{*})})\Big)\notag\\
&+\frac{\beta}{2}F_{1}\sum_{j=1}^{2}\frac{\partial\nu_{1}}{\partial\phi_{j}}(\phi+\Phi)(z_{j}-z_{j}^{(T_{*})})(\phi_{1}^{(T_{*})}+\phi_{2}^{(T_{*})})\notag\\
&+\frac{\beta}{2}F_{1}\sum_{j=1}^{2}\Big(\frac{\partial\nu_{1}}{\partial\phi_{j}}(\phi+\Phi)-\frac{\partial\nu_{1}}{\partial\phi_{j}}(\phi^{(T_{*})}+\Phi)\Big)
(z_{j}^{(T_{*})})(\phi_{1}^{(T_{*})}+\phi_{2}^{(T_{*})})\notag\\
&-F_{1}\sum_{j=1}^{2}\frac{\partial\nu_{1}}{\partial\phi_{j}}(\phi+\Phi)z_{j}(z_{1}-z_{1}^{(T_{*})})\\
&-F_{1}\sum_{j=1}^{2}\frac{\partial\nu_{1}}{\partial\phi_{j}}(\phi+\Phi)(z_{j}-z_{j}^{(T_{*})})z_{1}^{(T_{*})}\notag\\
&-F_{1}\sum_{j=1}^{2}\Big(\frac{\partial\nu_{1}}{\partial\phi_{j}}(\phi+\Phi)-\frac{\partial\nu_{1}}{\partial\phi_{j}}
(\phi^{(T_{*})}+\Phi)\Big)z_{j}^{(T_{*})}z_{1}^{(T_{*})}.
\end{align*}
Then we integrate it along the corresponding backward characteristic curve $t=t_{1}(x;\hat{t},\hat{x})$ defined by~\eqref{d10} and by~\eqref{a2}-\eqref{a3}, \eqref{b13}, \eqref{f1}-\eqref{f3}, \eqref{f6}-\eqref{f7} and~\eqref{f9}, we have
\begin{align*}
&|z_{1}(\hat{t},\hat{x})-z_{1}^{(T_{*})}(\hat{t},\hat{x})|\\
\leq&\frac{|\kappa_{1}|C_{S}^{*}\epsilon\xi^{N}}{F_{1}(\hat{t},\hat{x})}+C\epsilon C_{S}\epsilon\xi^{N}+C\epsilon C_{S}^{*}\epsilon\xi^{N}\\
&+\frac{F_{1}(\hat{t},\hat{x})-1}{F_{1}(\hat{t},\hat{x})}C_{S}^{*}\epsilon\xi^{N}+F_{1}(\hat{t},\hat{x})\mathop{\sup}\limits_{\phi\in\Psi}|\nabla\nu_{1}|
C_{R}C_{S}\epsilon\xi^{N}.
\end{align*}
Here we choose constant
$$C_{S}^{*}>30\mathop{\max}\limits_{1\leq i\leq2}\mathop{\sup}\limits_{\phi\in\Psi}|\nabla\nu_{i}|C_{R}C_{S}+30\beta_{*}A_{0}C_{S}+C_{S},$$
then we can get
\begin{align*}
|z_{1}(\hat{t},\hat{x})-z_{1}^{(T_{*})}(\hat{t},\hat{x})|\leq\hbar_{3}C_{S}^{*}\epsilon\xi^{N+1},\quad\forall(\hat{t},\hat{x})\in[(N+1)T_{0},\tau]\times[0,L],
\end{align*}
where constant $0<\hbar_{3}<1$.\\
\indent Similarly, we get
\begin{align*}
|z_{2}(\hat{t},\hat{x})-z_{2}^{(T_{*})}(\hat{t},\hat{x})|\leq\hbar_{3}C_{S}^{*}\epsilon\xi^{N+1},\quad\forall(\hat{t},\hat{x})\in[(N+1)T_{0},\tau]\times[0,L].
\end{align*}
Thus, by the arbitrariness of $\hat{t}\in[(N+1)T_{0},\tau]$, we have
\begin{align}
\|z_{i}(t,\cdot)-z_{i}^{(T_{*})}\|_{C^{0}}\leq\hbar_{3}C_{S}^{*}\epsilon\xi^{N+1},\quad\forall t\in[(N+1)T_{0},\tau].\label{f13}
\end{align}
Using~\eqref{c1} and~\eqref{d8}, we get
\begin{align*}
&w_{1}-w_{1}^{(T_{*})}\\
=&\Big(\nu_{1}(\phi+\Phi)-\nu_{1}(\phi^{(T_{*})}+\Phi)\Big)z_{1}-\nu_{1}(\phi^{(T_{*})}+\Phi)(z_{1}-z_{1}^{(T_{*})})\\
&+\frac{\beta}{2}\nu_{1}(\Phi)(\phi_{1}-\phi_{1}^{(T_{*})})+\frac{\beta}{2}\Big(\nu_{1}(\phi+\Phi)-\nu_{1}(\phi^{(T_{*})}+\Phi)\Big)(\phi_{1}+\phi_{2})\\
&+\frac{\beta}{2}\nu_{1}(\phi^{(T_{*})}+\Phi)\Big((\phi_{1}+\phi_{2})-(\phi_{1}^{(T_{*})}+\phi_{2}^{(T_{*})})\Big)\\
&+\frac{\beta}{2}\nu_{1}(\Phi)(\phi_{2}-\phi_{2}^{(T_{*})}),
\end{align*}
then by~\eqref{a2}, \eqref{b13}, \eqref{f1}-\eqref{f3} and~\eqref{f13}, we have
\begin{align}
&\|w_{1}(t,\cdot)-w_{1}^{(T_{*})}(t,\cdot)\|\notag\\
\leq&A_{0}\hbar_{3}C_{S}^{*}\epsilon\xi^{N+1}+C\epsilon C_{S}\epsilon\xi^{N}+2\beta_{*}A_{0}C_{S}\epsilon\xi^{N}\notag\\
\leq&A_{0}C_{S}^{*}\epsilon\xi^{N+1}.\label{f14}
\end{align}
This indicates that we complete the proof of~\theref{t5}.

\end{sloppypar}
\end{document}